\documentclass[11pt]{amsart} 

\usepackage[utf8]{inputenc} 

\usepackage{lipsum}

\usepackage[margin=1in]{geometry} 

\usepackage{graphicx} 

\usepackage[english]{babel}

\usepackage{booktabs} 
\usepackage{array} 
\usepackage{paralist} 
\usepackage{verbatim} 
\usepackage{subfig} 

\usepackage{caption}

\usepackage{amsfonts,amsmath,amssymb,bbm,amsthm,amsbsy,amscd}
\usepackage{hyperref}
\usepackage{tikz-cd}
\usepackage{mathtools}
\usepackage{bm}
\usepackage{pinlabel}

\usepackage{color}
\usepackage{pdfcolmk}

\usepackage{tikz}
\usepackage{tikz-cd}

\usepackage{IEEEtrantools}

\newenvironment{equ*}[1]{\begin{IEEEeqnarray*}{#1}}{\end{IEEEeqnarray*}}
\numberwithin{equation}{section}

\usepackage{bbm} 

\usepackage{cleveref}

\newenvironment{frcseries}{\fontfamily{frc}\selectfont}{}

\makeatletter                                                   
\newtheorem*{rep@theorem}{\rep@title}
\newcommand{\newreptheorem}[2]{%
	\newenvironment{rep#1}[1]{%
		\def\rep@title{#2 \ref{##1}}%
		\begin{rep@theorem}}%
		{\end{rep@theorem}}}
\makeatother

\theoremstyle{theorem}
\newtheorem{theoremA}{Theorem}

\newtheorem{thm}{Theorem}[section]
\newtheorem{lemma}[thm]{Lemma}

\crefname{lemma}{Lemma}{Lemmata}
\newtheorem{prop}[thm]{Proposition}
\newtheorem{cor}[thm]{Corollary}
\newtheorem{claim}[thm]{Claim}

\newtheorem*{thm*}{Theorem}
\newtheorem*{lemma*}{Lemma}
\newtheorem*{prop*}{Proposition}
\newtheorem*{corr*}{Corrolary}
\newtheorem*{claim*}{Claim}

\theoremstyle{remark}

\newtheorem*{rmk*}{Remark}
\newtheorem*{conj*}{Conjecture}
\newtheorem*{quest*}{Question}

\theoremstyle{definition}
\newtheorem{defn}[thm]{Definition}

\newtheorem*{defn*}{Definition}
\newtheorem*{exmp*}{Example}

\newreptheorem{theorem}{Theorem}
\newreptheorem{corollary}{Corollary}
\newreptheorem{proposition}{Proposition}

\newcommand{\R}{\mathbb{R}}

\newcommand{\Q}{\mathbb{Q}}
\newcommand{\Z}{\mathbb{Z}}
\newcommand{\N}{\mathbb{N}}
\newcommand{\Hbb}{\mathbb{H}}

\newcommand{\Isom}{\mathrm{Isom}}
\newcommand{\PSLtwo}{\mathrm{PSL}_2}

\newcommand{\Ends}{\mathrm{Ends}}

\newcommand{\att}{\texttt{a}}
\newcommand{\btt}{\texttt{b}}
\newcommand{\xtt}{\texttt{x}}
\newcommand{\ytt}{\texttt{y}}
\newcommand{\ztt}{\texttt{z}}

\newcommand{\area}{\mathrm{area}}
\newcommand{\len}{\mathrm{length}}

\newcommand{\cl}{\mathrm{cl}}
\newcommand{\scl}{\mathrm{scl}}

\newcommand{\col}{\colon}
\newcommand{\defeq}{\vcentcolon=}

\title{Uniform spectral gap of scl in $2$-orbifolds}

\author{Lvzhou Chen}
\address{Department of Mathematics\\ Purdue University\\ West Lafayette, Indiana, USA}
\email[L.~Chen]{lvzhou@purdue.edu}

\author{Nicolaus Heuer}
\address{Tudor Investment Corporation\\ Cambridge, England, UK}
\email[N.~Heuer]{nicolaus.heuer.maths@gmail.com}

\begin{document}
	
\begin{abstract}
	We show a \emph{uniform} spectral gap of stable commutator length for all compact hyperbolic $2$-orbifolds relative to the peripheral subgroups.
	Except for the case of a sphere with three cone points, we have an explicit uniform gap $1/36$.
	These estimates are needed in understanding stable commutator length in $3$-manifolds.
	Our methods use explicit quasimorphisms for the generic case, and use hyperbolic geometry (pleated surfaces) for the exceptional case of a sphere with three cone points.
\end{abstract}

	\maketitle
	
	
	\section{Introduction}
	This paper is a revised version of the appendix to \cite{CH:sclgapgog_old}, which was omitted when that article was shortened and revised for publication as \cite{CH:sclgapgog_new}.
	
	The stable commutator length (scl) in a group $G$ is a characteristic function $\scl_G: G\to [0,+\infty]$ that captures geometric and dynamical information of elements in the group; see Section~\ref{sec:background}.
	It is dual to quasimorphisms on the group by Bavard's duality \cite{bavard}.
	
	A rather common feature of scl is that groups $G$ with non-positive curvature typically have a \emph{(scl-)spectral gap} property: There is a constant $C=C(G)>0$ such that for all $g\in G$ either $\scl_G(g)\ge C$ or $\scl_G(g)=0$ (and such $g$ has a special form). For instance, this holds for all word-hyperbolic groups \cite{CF:sclhypgrp} and mapping class groups \cite{BBF},
	right-angled Artin or Coxeter groups \cite{Heuer,CH:RAAG_chain}.
	In a previous work, the authors showed that this also holds for $3$-manifold groups \cite{CH:sclgapgog_new}.
	The spectral gap property has connections to homomorphism rigidity, and obtaining sharp estimates of the gap $C$, especially for $C=1/2$, has seen applications to the study of generalized torsion elements \cite{gentorsion,CaiClay}.
	
	In this paper, we prove a \emph{uniform} spectral gap for all orbifold fundamental groups $\pi_1(B)$ of compact $2$-dimensional orbifolds $B$ with negative orbifold Euler characteristic $\chi_o(B)$. 
	We write $\scl_B$ for $\scl_{\pi_1(B)}$, and write $\scl_{(B,\partial B)}$ for stable commutator length \emph{relative} to the peripheral subgroups represented by the boundary components of $B$; see below for the definition. It is by definition that $\scl_B\ge\scl_{(B,\partial B)}$, and they are the same if $\partial B=\emptyset$.
	\begin{theoremA}\label{thmA: orbifold uniform gap}
		There is a uniform constant $C>0$, such that
		for any compact $2$-dimensional orbifold $B$ with $\chi_o(B)<0$, we have either $\scl_{(B,\partial B)}(g)\ge C$ or $\scl_{(B,\partial B)}(g)=0$ for all $g\in\pi_1(B)$.
		Moreover, if $B$ is not the $2$-sphere with three cone points, we can take $C$ to be $1/36$.
	\end{theoremA}
	
	Even if one is only interested in the closed case, it is essential to consider the case with boundary and relative scl. This is because such orbifolds occur naturally as vertex pieces
	in natural splittings of closed orbifold groups, for which we need to understand scl of vertex groups relative to edge groups. 
	In the case of nonempty boundary we obtain the stronger explicit estimate below.
	\begin{theoremA}\label{thmA: orbifold uniform gap with boundary}
		Let $B$ be a compact $2$-dimensional orbifold with nonempty boundary and $\chi_o(B)<0$.
		For any $g\in \pi_1(B)$, we have either
		$$\scl_{(B,\partial B)}(g)\ge 1/24$$
		or $\scl_{(B,\partial B)}(g)= 0$. Moreover, $\scl_{(B,\partial B)}(g) = 0$ if and only if $g$ has finite order, $g$ is conjugate to $g^{-1}$, or $g$ is peripheral.
	\end{theoremA}
	
	This is helpful for proving spectral gap results for Seifert fibered $3$-manifolds \cite{CH:sclgapgog_new}.
	Theorem~\ref{thmA: orbifold uniform gap with boundary} makes certain scl estimates explicit and uniform for Seifert fibered $3$-manifolds; 
	See \cite[Lemma~6.16 and Remark~6.17]{CH:sclgapgog_new}.
	
	The groups in Theorems~\ref{thmA: orbifold uniform gap} and \ref{thmA: orbifold uniform gap with boundary} are word-hyperbolic. 
	For each hyperbolic group, a spectral gap exists by Calegari--Fujiwara \cite{CF:sclhypgrp}, but the size of the gap in general depends on the number of generators 
	and the hyperbolicity constant $\delta$, and thus it is not uniform for an obvious reason and it is not explicit.
	Their theorem shows that there is also a spectral gap for \emph{integral chains}.
	In general, $\scl_G$ is defined for \emph{chains} $c=\sum_i c_i g_i$, a formal linear combination of elements $g_i\in G$ with $c_i\in \R$; see Section~\ref{sec:background} for details.
	We say it is an integral chain if each $c_i\in \Z$.
	However, as revealed by the authors in \cite{CH:RAAG_chain}, the size of the (optimal) gap for chains can be very different from that for elements: For right-angled Artin groups, there is a uniform spectral gap $1/2$ for elements, but the optimal gap for integral chains goes to zero for a sequence of such groups. The relative stable commutator length $\scl_{(G,\{G_\lambda\})}(g)$ is defined as the infimum of $\scl_G(g+\sum_\lambda c_\lambda)$ over all choices of chains $c_\lambda$'s in the subgroups $G_\lambda$'s; see Section~\ref{sec:background} for more details.
	
	When the orbifold $B$ has nonempty boundary, the orbifold group $\pi_1(B)$ is a free product of cyclic groups. 
	As a byproduct of our methods, we have the following explicit and uniform spectral gap for integral chains for all such free products, 
	where the equivalence of chains is explained in Section~\ref{sec:background}.
	The gap $1/12$ is sharp for this family of groups and agrees with the gap for elements; see \cite[Theorem~6.9]{CFL16} and \cite[Lemma 2.6]{CH:sclgapgog_new}.
	\begin{theoremA}\label{thmA: free product chain gap}
		Let $G$ be a free product of cyclic groups and let $c = \sum c_i g_i$ be an integral chain. 
		Then $\scl_G(c) \geq 1/12$ unless $c$ is equivalent to the zero chain, in which case $\scl_G(c)=0$.
	\end{theoremA}
	
	A powerful tool for proving lower bounds of scl is Bavard's duality \cite{bavard}.
	However, it is very difficult to use it to obtain good and explicit bounds due to our lack of understanding and limited constructions.
	Our proof of Theorems~\ref{thmA: orbifold uniform gap with boundary} and \ref{thmA: free product chain gap} uses a type of counting quasimorphisms that we define in free products of cyclic groups, which enjoy a separation property. 
	The exceptional case of a sphere with three cone points in Theorem~\ref{thmA: orbifold uniform gap} requires a different approach using hyperbolic geometry. 
	In particular, we provide the existence of pleated surfaces under soft assumptions in the appendix (Theorem~\ref{thm: restrict to pleated}), which seems hard to find in the literature.
	
%
%


\subsection*{Organization of the paper}
In Section~\ref{sec:background}, we recall the necessary background on stable commutator length, relative stable commutator length, quasimorphisms, Bavard's duality, and $2$-orbifolds. 
In Section~\ref{sec: chain gap}, we construct counting quasimorphisms in free products of cyclic groups and develop their basic properties to prove Theorem~\ref{thmA: free product chain gap}.
Then we apply these quasimorphisms to prove the relative spectral gap for orbifolds with nonempty boundary (Theorem~\ref{thmA: orbifold uniform gap with boundary}) in Section~\ref{sec: orbifold with boundary}.
Finally we deal with closed orbifolds and prove Theorem~\ref{thmA: orbifold uniform gap} in Section~\ref{sec: closed case}: 
the generic closed case is handled using acylindrical Bass--Serre tree actions, while the case of a sphere with three cone points is treated separately using hyperbolic geometry. 
The appendix establishes the existence of pleated representatives for admissible surfaces, which is used in the exceptional closed case.

\section{Background} \label{sec:background}
	

\subsection{scl and relative scl}\label{subsec: scl and rel scl}
We recall some basics about stable commutator length. A comprehensive reference is \cite{Cal:sclbook}.
\begin{defn}
	For any group $G$ and any element $g$ in the commutator subgroup $[G,G]$, the commutator length $\cl_G(g)$ is the minimal number of commutators whose product is $g$.
	The stable commutator length $\scl_G(g)$ is the limit
	$$\scl_G(g)\defeq \lim_n \frac{\cl_G(g^n)}{n},$$
	which always exists by subadditivity. 
\end{defn}

We will use the topological definition below, which also extends the definition above to \emph{chains} in $G$. 
Recall that a chain in $G$ is a formal linear combination $c=\sum_i c_i g_i$ of finitely many elements $g_i\in G$, with $c_i\in \R$.
We say the chain $c$ is integral (resp. rational) if $c_i\in \Z$ (resp. $c_i\in\Q$) for all $i$.
Denote by $C_1(G)$ the $\R$-vector space of all chains in $G$.
Taking homology gives a linear map $C_1(G)\to H_1(G)$, whose kernel $B_1(G)$ consists of null-homologous chains.

\begin{defn}\label{def: admissible surfaces}
	Fix any path-connected space $X$ with $\pi_1(X)=G$, and fix a rational chain $c=\sum_i c_i g_i$.
	An \emph{admissible surface} for $c$ of degree $n$ is a map $f:S\to X$ from a compact oriented surface $S$ such that
	the image of each boundary component represents the conjugacy class of some $g_i^k$ for some $k\in\Z$, 
	and the sum of $kg_i$ over all boundary components is equal to the chain $n\cdot c$. 
	We say $S$ is \emph{monotone} if each $k>0$.
	Such (monotone) admissible surfaces exist for some $n$ exactly when $c$ is null-homologous (i.e. $c\in B_1(G)$).
	Although the map $f$ is an important part of the data, we often suppress it and simply write $S$ for an admissible surface and denote the degree as $n(S)$.
	
	Let $\chi^-(S)$ be the Euler characteristic of $S$ after removing $S^2$ or $D^2$ components. We measure the complexity of an admissible surface $S$ by $-\chi^-(S)/2n(S)$. 
	For $c\in B_1(G)$ we define
	$$\scl_G(c)\defeq \inf_S \frac{-\chi^-(S)}{2n(S)},$$
	where the infimum is taken over all admissible surfaces $S$ for the chain $c$.
	This agrees with the algebraic definition when $c=g\in [G,G]$; see \cite[Proposition 2.10]{Cal:sclbook}.
	Moreover, we can restrict to monotone admissible surfaces (\cite[Proposition 2.13]{Cal:sclbook}), which is what we do in the sequel.
	It is a fact that $\scl_G$ extends continuously in a unique way to a semi-norm on $B_1(G)$; see \cite[Section 2.6]{Cal:sclbook}.
	If $[c]\neq 0\in H_1(G;\R)$, then we define $\scl_G(c)\defeq +\infty$.	
\end{defn}

The basic properties below are well known and easily follow from the definition.
\begin{lemma}\label{lemma: scl basic prop}
	For any group $G$ and $g\in G$, we have
	\begin{enumerate}
		\item $\scl_G(g^n)=n\cdot \scl_G(g)$ for all $n\in \Z_+$, and in particular $\scl_G(g)=0$ for any $g$ of finite order;
		\item $\scl_G(g)\ge\scl_H(f(g))$ for any homomorphism $f:G\to H$;
		\item $\scl_G(g)$ is invariant under automorphisms of $G$, and in particular it is conjugation invariant;
		\item $\scl_G(g)=0$ if $g^n$ is conjugate to $g^{-n}$ for some $n\in\Z_+$.
	\end{enumerate}
\end{lemma}

We consider chains up to the equivalence relation given by the subspace $H(G)\le B_1(G)$ spanned by chains of the forms below:
\begin{enumerate}
	\item $g^n-n\cdot g$, for any $g\in G$ and $n\in \Z$;
	\item $hgh^{-1}-g$, for any $g,h\in G$.
\end{enumerate}
Equivalent chains have the same scl \cite[Section 2.6]{Cal:sclbook}, so scl is a well defined semi-norm on the quotient $B_1^H(G)\defeq B_1(G)/H(G)$.
Up to equivalence, we can always write a chain as $c=\sum_i c_i g_i$ with each $c_i>0$ and no $g_i$ is conjugate to a (positive or negative) power of $g_j$ for $j\neq i$.

Scl of chains appears naturally in the index formula \cite[Corollary 2.81]{Cal:sclbook} below about scl in finite index subgroups.
\begin{lemma}[Index formula]\label{lemma: index formula}
	Let $\Gamma$ be a finite index normal subgroup of $G$ and denote the quotient by $H=G/\Gamma$. Then for any $\gamma\in \Gamma$ we have
	$$\scl_G(\gamma)=\frac{1}{|H|}\cdot \scl_\Gamma(\sum_{h\in H}h\gamma h^{-1}).$$
\end{lemma}

Now we turn to relative stable commutator length; see \cite{Chen:sclBS,CH:sclgapgog_new} for more details. 
\begin{defn}\label{def: relative scl}
	For a group $G$ and a collection of conjugacy classes of subgroups $\{G_\lambda\}_{\lambda\in\Lambda}$, the relative scl is defined as
	$$\scl_{(G,\{G_\lambda\})}(c)=\inf_{c_\lambda} \scl_G(c+\sum_\lambda c_\lambda),$$
	where the infimum is taken over all choices of chains $c_\lambda \in C_1(G_\lambda)$ across $\lambda\in\Lambda$.
\end{defn}

Note that it suffices to consider those choices of $c_\lambda$'s such that $c+\sum c_\lambda$ is null-homologous in $G$, as otherwise the scl value is $+\infty$.

The basic properties below easily follow from the definition and basic properties of scl.
\begin{lemma}\label{lemma: relative scl basic prop}
	Fix $G$ and $\{G_\lambda\}_{\lambda\in\Lambda}$ as above.
	\begin{enumerate}
		\item $\scl_{(G,\{G_\lambda\})}(g^n)=n\cdot \scl_{(G,\{G_\lambda\})}(g)$ for all $g\in G$;
		\item $\scl_{(G,\{G_\lambda\})}(g)=0$ if some nonzero power of $g$ is conjugate into some $G_\lambda$;
		\item For any chain $c\in C_1(G)$, we have $\scl_G(c)\ge\scl_{(G,\{G_\lambda\})}(c)$. Thus $\scl_{(G,\{G_\lambda\})}(c)=0$ if $\scl_G(c)=0$.
		In particular, $\scl_{(G,\{G_\lambda\})}(g)=0$ for any $g\in G$ with $g^n$ conjugate to $g^{-n}$ for some $n\in\Z_+$.
	\end{enumerate}
\end{lemma}

Relative scl can also be computed using \emph{relative} admissible surfaces; see \cite[Proposition 2.9]{Chen:sclBS}.
\begin{prop}\label{prop: relative adm surf}
	Fix any path-connected space $X$ with $\pi_1(X)=G$, and a collection of path-connected subspaces $X_\lambda$ with $\pi_1(X_\lambda)$ conjugate to $G_\lambda$ for each $\lambda\in \Lambda$.
	A relative admissible surface $S$ of degree $n(S)$ for a rational chain $c=\sum_i c_i g_i$ is an admissible surface of degree $n(S)$ for some chain $c+\sum_\lambda c_\lambda$ with $c_\lambda\in C_1(G_\lambda)$. Intuitively, these are surfaces that are admissible for $c$ except that there are some extra boundary components that are mapped to $X_\lambda$ (up to homotopy).
	Then the following infimum over all relative admissible surfaces gives relative scl:
	$$\scl_{(G,\{G_\lambda\})}(c)=\inf_S \frac{-\chi^-(S)}{2n(S)}.$$
\end{prop}


\subsection{Quasimorphisms and Bavard's duality}
\begin{defn}\label{def: quasimorphisms}
	A function $\phi: G\to \R$ is a \emph{quasimorphism} if
	$$D(\phi)\defeq \sup_{g,h\in G} |\phi(g)+\phi(h)-\phi(gh)|<\infty,$$
	in which case we call $D(\phi)$ the \emph{defect} of $\phi$.
	We say $\phi$ is \emph{homogeneous} if $\phi(g^n)=n\cdot \phi(g)$ for all $g\in G$ and $n\in\Z$.
	Denote by $Q(G)$ the $\R$-vector space of homogeneous quasimorphisms.
\end{defn}

Homogeneous quasimorphisms are always conjugate invariant, i.e. $\phi(hgh^{-1})=\phi(g)$ for all $g,h\in G$.

One can always homogenize a quasimorphism to obtain a homogeneous quasimorphism.
\begin{lemma}\label{lemma: homogenization}
	For any quasimorphism $\phi: G\to \R$, the homogenization $\bar{\phi}$ defined as
	$$\bar\phi(g)\defeq \lim_n \frac{\phi(g^n)}{n}$$
	is a well defined homogeneous quasimorphism. Moreover, we have $D(\bar{\phi})\le 2D(\phi)$.
\end{lemma}
We sometimes write $\bar\phi$ to denote a homogeneous quasimorphism, to emphasize the fact that it is homogeneous, 
even though it is not directly defined as the homogenization of some quasimorphism $\phi$.

We have a pairing $\phi(c)\defeq\sum c_i\phi(g_i)$ for any chain $c=\sum c_i g_i$ and a homogeneous quasimorphism $\phi\in Q(G)$.
The Bavard duality shows that this gives a duality between scl and homogeneous quasimorphisms.
\begin{thm}[Bavard's duality \cite{bavard}]\label{thm: Bavard}
	Fix a group $G$ and any chain $c\in B_1(G)$, we have
	$$\scl_G(c)=\sup_{\phi\in Q(G)} \frac{|\phi(c)|}{2D(\phi)}.$$
	In particular, we have $\scl_G(c)\ge \frac{|\phi(c)|}{2D(\phi)}$ for any $\phi\in Q(G)$.
\end{thm}
We will only use the inequality above, which is the easy direction.

\subsection{$2$-dimensional Orbifolds}
Here we consider compact $2$-dimensional orbifolds $B$ with only finitely many cone points as singularities.
Topologically, the underlying space is a compact surface $S=S(B)$, not necessarily orientable.
In particular, the genus and the number of boundary components of $B$ are defined just as for the underlying surface $S$.
In addition, in the interior of $S$, there are finitely many \emph{cone points} of orders $o_1,\ldots, o_n\ge2$ respectively, where we think of a neighborhood of each such point
as modeled on $D^2$ mod the standard $\Z/o_i$ rotation action, $1\le i\le n$.
The \emph{orbifold Euler characteristic} is defined as
$$\chi_o(B)=\chi(S(B))-\sum_i (1-\frac{1}{o_i}).$$

We focus on orbifolds $B$ with $\chi_o(B)<0$, as the other cases have amenable (actually virtually abelian) orbifold fundamental group $\pi_1(B)$, which implies that $\scl_{\pi_1(B)}$ vanishes.
When $\chi_o(B)<0$, each such orbifold has a finite orbifold cover that is an honest surface, or equivalently, it can be realized as a compact surface (with negative Euler characteristic) modulo a finite group action.
Moreover, $B$ has a hyperbolic structure, that is, we can think of $B$ as a quotient of the hyperbolic plane by a discrete subgroup $\pi_1(B)$ of the isometry group $\Isom(\Hbb^2)$.

We are interested in $\scl_{(B,\partial B)}$, which is our shorthand of scl of $\pi_1(B)$ \emph{relative} to the collection of peripheral cyclic subgroups. We also write $\scl_B$ for $\scl_{\pi_1(B)}$.

These orbifold groups have very explicit presentations and splittings (as graphs of groups) just like the fundamental groups of compact surfaces, except that the cone points give elements of finite order.
We will recall the relevant presentations or splittings when we need them.

\section{Spectral gaps for chains in free products of cyclic groups}\label{sec: chain gap}
In this section we will prove some key results (Lemmas~\ref{lemma: Brooks} and \ref{lemma: cyclic conjugate minimal}) regarding certain counting quasimorphisms, which are essential in our proof of the spectral gap of orbifolds relative to the boundary. 
As a byproduct, we use such quasimorphisms to prove a uniform spectral gap for chains in free products of cyclic groups (Theorem~\ref{thmA: free product chain gap}). 
This is a generalization of \cite{tao} and uses analogous arguments.

Suppose $G = Z_1 \star \cdots \star Z_m \star (\Z/o_1) \star \cdots \star (\Z/o_n)$ with $Z_i \cong \Z$. Let $\xtt_1, \ldots, \xtt_m$ be the generators of $Z_1 \star \cdots \star Z_m$ and let $\ytt_1, \ldots, \ytt_n$ be the generators of $(\Z/o_1) \star \cdots \star (\Z/o_n)$ with orders $o_1, \ldots, o_n$.
Consider the alphabet $S = \{ \xtt_1^{\pm 1}, \ldots, \xtt_m^{\pm 1}, \ytt_1^{i_1}, \ldots, \ytt_n^{i_n} \}$ where $0 < i_j < o_j$
for all $j \in \{1, \ldots, n \}$. We say a word $w$ is reduced if 
no $\xtt_j$ is adjacent to $\xtt_j^{-1}$ and no $\ytt_j^{i_1}$ is adjacent to $\ytt_j^{i_2}$.
Denote the word length of $w$ by $|w|$, and for any $g\in G$ let $|g|=|w|$ for any cyclically reduced representative $w$.

Every element $g \in G$ can be uniquely represented by a reduced word $\bar{g}$ in the alphabet $S$. For two elements $g,h \in G$ we define $C_g(h)$ to be the number of times that $\bar{g}$ occurs as a subword of $\bar{h}$. Finally we set $\phi_g(h) \defeq C_g(h) - C_{g^{-1}}(h)$ analogously to the Brooks quasimorphism first defined by Brooks in \cite{brooks}; see also \cite[Section 2.3.2]{Cal:sclbook}. Note that $C_g(h) = C_{g^{-1}}(h^{-1})$.

A word $w$ in $S$ is \emph{self-overlapping} if there are reduced words $u, v$, where $v$ is non-trivial, such that $w = v u v$ as a reduced word. Note that any periodic word $w$ (i.e. $w=v^k$ as a reduced word for $k\ge2$) is self-overlapping.

\begin{lemma}\label{lemma: Brooks}
Suppose the reduced word $\bar{g}$ representing $g \in G$ is cyclically reduced of length at least $2$
and not self-overlapping. Then
\begin{enumerate}
\item $\phi_g \col G \to \Z$ is a quasimorphism with $D(\phi_g) \leq 3$ and its homogenization satisfies $D(\bar{\phi}_g) \leq 6$.\label{item:defect}
\item If $g$ is not conjugate to $g^{-1}$ then $\bar{\phi}_g(g) = 1$.\label{item:homogeneous}
\item For every cyclically reduced word $w$ with $|w| \leq |g|$ representing the conjugacy class of $h\in G$, we have that either $w$ is a cyclic conjugate of $\bar{g}$ or $\bar{g}^{-1}$, or $\bar{\phi}_g(h) = 0$.\label{item:cyclic conjugates}
\end{enumerate} 
\end{lemma}

This lemma is analogous to \cite[Lemmas 3.1 and 3.2]{tao}, which deals with the torsion-free case. In the absence of torsion one may deduce that $D(\phi_g) \leq 2$, yielding $D(\bar{\phi}_g) \leq 4$.

\begin{proof}
Observe that if $v, w$ are reduced words and $\ztt \in S \cup \{ \emptyset \}$ such that $v^{-1} \cdot \ztt \cdot  w$ is reduced then $C_g(v^{-1} \cdot \ztt \cdot w) = C_g(v^{-1})+C_g(w) + C_g((v')^{-1} \cdot \ztt \cdot w')$, where $v'$ is the prefix of length $|g|-1$ of $v$ if $|v| \geq |g|$ and $v' = v$ otherwise, and $w'$ is defined analogously.

For every $h_1, h_2 \in G$ there are words $c_1, c_2, c_3$ in $G$ and letters $\ztt_1, \ztt_2, \ztt_3 \in S \cup \{ \emptyset \}$ 
such that
\begin{eqnarray*}
h_1 &=& c_1^{-1} \ztt_1 c_2 \\
h_2 &=& c_2^{-1} \ztt_2 c_3 \\
(h_1 h_2)^{-1} &=& c_3^{-1} \ztt_3 c_1
\end{eqnarray*}

To see bullet (\ref{item:defect}), using the notation above, note that
\begin{eqnarray*}
\phi_g(h_1) &=&
C_g(c_1^{-1}) + C_g(c_2) + C_g((c_1')^{-1} \ztt_1 c_2') - C_{g^{-1}}(c_1^{-1}) - C_{g^{-1}}(c_2) - C_{g^{-1}}((c_1')^{-1} \ztt_1 c_2') \\
\phi_g(h_2) &=&
C_g(c_2^{-1}) + C_g(c_3) + C_g((c_2')^{-1} \ztt_2 c_3') - C_{g^{-1}}(c_2^{-1}) - C_{g^{-1}}(c_3) - C_{g^{-1}}((c_2')^{-1} \ztt_2 c_3') \\
\phi_g((h_1 h_2)^{-1}) &=&
C_g(c_3^{-1}) + C_g(c_1) + C_g((c_3')^{-1} \ztt_3 c_1') - C_{g^{-1}}(c_3^{-1}) - C_{g^{-1}}(c_1) - C_{g^{-1}}((c_3')^{-1} \ztt_3 c_1').
\end{eqnarray*}

Using that $C_g(u) = C_{g^{-1}}(u^{-1})$ we see that
\begin{eqnarray*}
\phi_g(h_1) + \phi_g(h_2) - \phi_g(h_1 h_2) &=& 
\phi_g(h_1) + \phi_g(h_2) + \phi_g((h_1 h_2)^{-1}) \\
&=& C_g((c_1')^{-1} \ztt_1 c_2') - C_{g^{-1}}((c_1')^{-1} \ztt_1 c_2') \\
&+& C_g((c_2')^{-1} \ztt_2 c_3') - C_{g^{-1}}((c_2')^{-1} \ztt_2 c_3') \\
&+& C_g((c_3')^{-1} \ztt_3 c_1') - C_{g^{-1}}((c_3')^{-1} \ztt_3 c_1').
\end{eqnarray*}
All of the terms $(c_j')^{-1} \ztt_j c_{j+1}'$ have word length strictly less than $2 |g|$. 
As $\bar{g}$ is not self-overlapping, we conclude that $C_g((c_j')^{-1} \ztt_j c_{j+1}') \in \{ 0, 1 \}$,
and similarly $C_{g^{-1}}((c_j')^{-1} \ztt_j c_{j+1}') \in \{ 0, 1 \}$.
Thus each row in the expression above has absolute value at most $1$, and
$$
|\phi_g(h_1) + \phi_g(h_2) - \phi_g(h_1 h_2)| \leq 3.
$$ 
Hence $D(\phi_g)\le 3$ and $D(\bar{\phi}_g) \leq 2 D(\phi_g) \leq 6$ by Lemma~\ref{lemma: homogenization}.
This shows bullet (\ref{item:defect}).

For bullet (\ref{item:homogeneous}), 
observe that $C_g(g^n) = n$, as $g$ is not self-overlapping and is cyclically reduced of length at least $2$.
Suppose that $C_{g^{-1}}(g^n) \neq 0$. 
As $\bar{g}$ and $\bar{g}^{-1}$ have the same length, observe that $\bar{g}^{-1}$ must already be a subword of $\bar{g} \cdot \bar{g}$ in this case. We may write $\bar{g} \cdot \bar{g} = u_1 \bar{g}^{-1} u_2$ as a reduced expression. Since $\bar{g} \cdot \bar{g}$ is $|g|$-periodic, we conclude that $g = u_1 u_2$ and $g^{-1} = u_2 u_1$, thus $g$ and $g^{-1}$ are conjugate. 

To see bullet (\ref{item:cyclic conjugates}), assume that $w$ is any cyclically reduced word with $|w| \leq |g|$ representing $h\in G$.
If $|w| = |g|$, then by the same argument as above we see that if $C_g(w^n) \neq 0$ for some $n\ge1$, then $\bar{g}$ is a cyclic conjugate of $w$, and similarly if $C_{g^{-1}}(w^n) \neq 0$.
If $|w| < |g|$, suppose that $C_g(w^n) \neq 0$ for some integer $n \in \N$. Because $|w| < |g|$ we have that $n \geq 2$. Then we may write $g = w_e w^m w_s$ as a reduced word where $w_e$ is some suffix of $w$ and $w_s$ is some prefix of $w$ and where $m \geq 0$. Note that $g$ is not a proper power since $\bar{g}$ is cyclically reduced and not self-overlapping. So either $w_e$ or $w_s$ is nontrivial.
If $|w_s| + |w_e| \leq |w|$, then $m\ge1$ and we may write $g = w_e w_s w' w^{m'} w'' w_e w_s$ for an appropriate choice of subwords $w', w''$ and an integer $m'\ge0$. 
This contradicts the fact that $\bar{g}$ is not self-overlapping.
If $|w_s| + |w_e| > |w|$, then the prefix $w_s$ and the suffix $w_e$ of $w$ intersect in some non-trivial word $w_0$. Thus we may write $w_e = w_0 w_e'$ and $w_s = w_s' w_0$. 
Hence $g = w_0 w_e' w^m w_s' w_0$, which again contradicts the fact that $\bar{g}$ is not self-overlapping. 
This shows that $\bar{\phi}_g(h)=\bar{\phi}_g(w)=0$ for all such $w$ unless $w$ is a cyclic conjugate of $\bar{g}$ or $\bar{g}^{-1}$.
\end{proof}

\begin{lemma} \label{lemma: cyclic conjugate minimal}
Every element $g \in G$ that is not a proper power is conjugate to an element $g' \in G$ such that $\bar{g}'$ is cyclically reduced 
and not self-overlapping.
\end{lemma}
\begin{proof}
We follow the strategy of \cite[Lemmas 3.1 and 3.2]{tao}. Pick an arbitrary total order $\prec$ on $S$ and extend this to a lexicographic order on reduced words and thus on $G$. Let $g'$ be a conjugate of $g$ such that $\bar{g}'$ is cyclically reduced 
and such that it is minimal with respect to $\prec$ among all conjugates of $g$ with this property.
We claim that $\bar{g}'$ is not self-overlapping.

If not, then we may write $\bar{g}' = u v u$ as a reduced expression, where $u$ is non-trivial. Then $|uv|=|vu|$ and $uv\neq vu$ as otherwise $uv$ and $g$ must be proper powers, contradicting our assumption. If $uv\prec vu$ then $u u v$ is a cyclic conjugate with $u u v\prec u v u$, which contradicts our choice of $\bar{g}'$. Similarly if $vu\prec uv$ then $v u u$ is a cyclic conjugate with $v u u\prec u v u$, which also contradicts our choice of $\bar{g}'$.
\end{proof}

Now we use such quasimorphisms to deduce Theorem~\ref{thmA: free product chain gap}, which we restate as follows:
\begin{thm}\label{thm:gap on chains of free products of cyclic groups}
	Let $G$ be a free product of cyclic groups and let $c = \sum c_i g_i$ be an integral chain that is nontrivial in $B_1^H(G)$. Then $\scl_G(c) \geq 1/12$. This gap is sharp on the class of free products of cyclic groups.
\end{thm}
To see that this gap is sharp, let $G = (\Z/2) \star (\Z/3)$ be the free product of the cyclic group of order $2$ with generator $\att$ and the cyclic group of order $3$ with generator $\btt$, then $\scl_G(\att \btt) = 1/12$ by the product formula \cite[Theorem 2.93]{Cal:sclbook}.
\begin{proof}
Let $c=\sum_i c_i g_i$ be an integral chain where all $c_i$ are non-zero integers.
Up to replacing $c$ by an equivalent integral chain using the defining relation of $B_1^H(G)$ (see Section \ref{subsec: scl and rel scl}), we may assume that every $\bar{g_i}$ is cyclically reduced, not proper powers, and none of the $g_i$'s are conjugate to each other or their inverses.

Without loss of generality, assume that $|g_1|$ is maximal among all $|g_i|$. We also assume $|g_1|\ge2$, since otherwise being null-homologous implies that $c$ must be equivalent to the zero chain.
By Lemma \ref{lemma: cyclic conjugate minimal} there is a cyclic conjugate $g'$ of $g_1$ such that $g'$ is cyclically reduced and not self-overlapping.

Set $\phi = \bar{\phi}_{g'}$. Since $g'$ is conjugate to $g_1$ and $\phi$ is homogeneous, we have $\phi(g_1) = 1$. For every $i>1$ we have $|g_i|\le|g_1|$ and $\phi(g_i) = 0$ by Lemma \ref{lemma: Brooks}. By Bavard's duality Theorem~\ref{thm: Bavard} and the fact that $D(\phi)\le 6$, we deduce that
$$
\scl_G(c) \geq \frac{|\phi(c)|}{2 D(\phi)}= 
\frac{|\sum_i c_i \phi(g_i)|}{2 D(\phi)} = 
\frac{|c_1|}{2D(\phi)} \geq \frac{1}{12}.
$$
\end{proof}

For the special case where $c$ is a single element $g\in G$ that is not conjugate into any free factor, it is not equivalent to the zero chain provided that $g$ is not conjugate to $g^{-1}$. So we have the following corollary, which also follows from \cite[Theorem 6.9]{CFL16} (see also \cite[Lemma 2.6]{CH:sclgapgog_new}).
\begin{cor}\label{cor: free product gap for elements}
	Let $G$ be a free product of cyclic groups. Suppose $g\in G$ is not conjugate into any free factor, then $\scl_G(g)\ge1/12$ or $\scl_G(g)=0$, and the latter case occurs exactly when $g$ is conjugate to $g^{-1}$.
\end{cor}

\section{Spectral gap of orbifolds groups relative to the boundary}\label{sec: orbifold with boundary}

In this section we use the quasimorphisms constructed in the previous section to give bounds on the stable commutator length of elements in orbifold groups relative to the boundary. 
The goal is to prove Theorem~\ref{thmA: orbifold uniform gap with boundary}, which we restate as follows.

\begin{thm}\label{thm: orbifold rel gap}
	Let $B$ be an orbifold with nonempty boundary. 
	Then for any $g\in \pi_1(B)$, we have either $\scl_{(B,\partial B)}(g)\ge 1/24$ or $\scl_{(B,\partial B)}(g)=0$.
	Moreover, $\scl_{(B,\partial B)}(g)=0$ if and only if
	\begin{enumerate}
		\item $g$ has finite order,\label{item: vanish if torsion}
		\item $g$ is conjugate to $g^{-1}$, or\label{item: vanish if conjugate to inv}
		\item $g$ is represented by a boundary loop.\label{item: vanish on the boundary}
	\end{enumerate}
\end{thm}

First consider the orientable case. Let $B$ be a $2$-dimensional orientable orbifold of genus $k$ with $m+1$ boundary components and $n$ cone points of orders $o_1, \ldots, o_n \in \N$, where $k,m\ge 0$. 
Then the orbifold fundamental group $G=\pi_1(B)$ is
$Z_1 \star\cdots \star Z_{2k}\star Z^b_1\star \cdots \star Z^b_{m} \star (\Z/o_1) \star \cdots \star (\Z/o_n)$ where all $Z_\bullet^\bullet$'s are infinite cyclic groups.
Let $\xtt_1, \ldots, \xtt_{2k}$ be the natural generating set of  $Z_1 \star \cdots \star Z_{2k}$, let $\btt_1, \ldots, \btt_{m}$ be the natural generating set of $Z^b_1\star \cdots \star Z^b_{m}$, and let $\ytt_1, \ldots, \ytt_n$ be the natural generating set of $(\Z/o_1) \star \cdots \star (\Z/o_n)$.
With an appropriate choice, the $m+1$ boundary components with the induced orientation are represented by
$$ -\btt_1, -\btt_2, \cdots, -\btt_m, [\xtt_1, \xtt_2] \cdots [\xtt_{2k-1}, \xtt_{2k}] \btt_1 \cdots \btt_m \ytt_1 \cdots \ytt_n.$$
Then the boundary chain $\partial B$ is their sum
$$
\partial B \defeq [\xtt_1, \xtt_2] \cdots [\xtt_{2k-1}, \xtt_{2k}] \btt_1 \cdots \btt_m \ytt_1 \cdots \ytt_n - (\btt_1 + \cdots + \btt_m).
$$

\begin{proof}[Proof of Theorem \ref{thm: orbifold rel gap} for orientable orbifolds]
	It is obvious from Lemma~\ref{lemma: relative scl basic prop} that $\scl_{(B,\partial B)}(g)=0$ in cases (\ref{item: vanish if torsion}), (\ref{item: vanish if conjugate to inv}) and (\ref{item: vanish on the boundary}). 
	Assuming none of these cases occur, we now show $\scl_{(B,\partial B)}(g)\ge 1/24$. 
	
	Identify $G=\pi_1(B)$ with the free product of cyclic groups above. Note that the $m+1$ boundary elements (as homology classes) span an $m$-dimensional subspace $V_{\partial B}$ in $H_1(\pi_1(B);\R)$, and their only linear combinations that are null-homologous must be of the form $t\cdot \partial B$ for some $t\in\R$. 

	We assume that $[g]$ lies in $V_{\partial B}$ since otherwise $\scl_{(B,\partial B)}(g)=\infty$. Then there is a unique chain $c=\sum c_i \btt_i$ for some $c_i\in\R$ such that $[g]=[c]$, and $\scl_{(B,\partial B)}(g)=\inf_{t\in\R}\scl_B(g-c+t\partial B)$. Note that our assumptions imply that $g$ is not conjugate into any free factor.

	Write $g = h^q$ with $q\in\Z_+$ so that $h$ is not a proper power. Let $h'$ be the conjugate of $h$ as in Lemma \ref{lemma: cyclic conjugate minimal}. Note that $h'$ and $(h')^{-1}$ are not conjugate since $g$ and $g^{-1}$ are not, and that $h'$ has length at least $2$ since $g$ is not conjugate into any free factor. Thus $\bar{\phi}_{h'}(g)= q \geq 1$ and $\bar{\phi}_{h'}(\btt_i)=0$ for all $i$ by Lemma \ref{lemma: Brooks}, where $\bar{\phi}_{h'}$ is the Brooks quasimorphism constructed in the previous section. In particular, we have $\bar{\phi}_{h'}(c)=0$ and $\bar{\phi}_{h'}(\partial B)=\bar{\phi}_{h'}(b)$, where $b$ is the only long word in the chain $\partial B$, i.e.
	$$
	b \defeq [\xtt_1, \xtt_2] \cdots [\xtt_{2k-1}, \xtt_{2k}] \btt_1 \cdots \btt_m \ytt_1 \cdots \ytt_n.
	$$
		
	If $\bar{\phi}_{h'}(b) = 0$ then by Bavard's Duality and the estimate $D(\bar{\phi}_{h'})\le 6$ from Lemma \ref{lemma: Brooks} we have
	$$
	\scl_B(g -c + t\cdot \partial B) \geq \frac{|\bar{\phi}_{h'}(g -c + t\cdot \partial B)|}{2 D(\bar{\phi}_{h'})} \geq \frac{|\bar{\phi}_{h'}(g)|}{12} \geq \frac{1}{12}.
	$$

	We are left with the case where $\bar{\phi}_{h'}(b) \neq 0$. By Lemma \ref{lemma: Brooks} and the assumption that $g$ is not boundary parallel, this implies $|b|>|h'|\ge 2$ and moreover either $h'$ or its inverse is cyclically a proper subword of $b$. Suppose without loss of generality that $h'$ (instead of $h'^{-1}$) appears in $b$ up to cyclic permutation.

	From the expression of $b$ we observe that it can only contain one copy of $h'$, and there is no occurrence of $h'^{-1}$.
	This implies $\bar{\phi}_{h'}(\partial B)=\bar{\phi}_{h'}(b)=1$. 
	Observe also that $b$ is not self-overlapping and has length at least $2$, 
	thus $\bar{\phi}_{b}$ is a quasimorphism of defect at most $6$, $\bar{\phi}_{b}(b)=1$ (as $b$ is not conjugate to $b^{-1}$ since $\chi_o(B)<0$), and $\bar{\phi}_{b}(h')=0=\bar{\phi}_{b}(\btt_i)$ by Lemma \ref{lemma: Brooks}. Hence $\bar{\phi}_{b}(g)=q\bar{\phi}_{b}(h')=0$ and $\bar{\phi}_{b}(c)=0$.
	
	Set $\bar{\phi} \defeq \bar{\phi}_{h'} - \bar{\phi}_{b}$. We have the crude estimate $D(\bar{\phi})\leq  D(\bar{\phi}_{h'})+ D(\bar{\phi}_{b}) \leq 12$.
	Also note that $\bar{\phi}(\partial B)=\bar{\phi}_{h'}(b)-\bar{\phi}_{b}(b) = 1-1= 0$ and
	$\bar{\phi}(g) = \bar{\phi}_{h'}(g)=q\bar{\phi}_{h'}(h')\ge1$.
	Thus
	$$
	\scl_B(g -c + t \partial B) \geq \frac{\bar{\phi}(g -c + t \partial B)}{2 D(\bar{\phi})}=\frac{\bar{\phi}(g)}{2 D(\bar{\phi})} \geq \frac{1}{24},
	$$
	which finishes the proof of the orientable case.
\end{proof}

The nonorientable case is similar. In this case, if $B$ has (nonorientable) genus $k$ and $m+1$ boundary components with cone points of order $o_1,\ldots, o_n$, where $k\ge1,m\ge0$, then $\pi_1(B)$ is
$Z_1 \star\cdots \star Z_{k}\star Z^b_1\star \cdots \star Z^b_{m} \star (\Z/o_1) \star \cdots \star (\Z/o_n)$ where all $Z_\bullet^\bullet$'s are infinite cyclic groups.
Let $\xtt_1, \ldots, \xtt_{k}$ be the natural generating set of  $Z_1 \star \cdots \star Z_{k}$, let $\btt_1, \ldots, \btt_{m}$ be the natural generating set of $Z^b_1 \cdots \star Z^b_{m}$, and let $\ytt_1, \ldots, \ytt_n$ be the natural generating set of $(\Z/o_1) \star \cdots \star (\Z/o_n)$.
With an appropriate choice, the $m+1$ boundary components are represented by
$$ -\btt_1, -\btt_2, \cdots, -\btt_m, \xtt_1^2\cdots \xtt_k^2 \btt_1 \cdots \btt_m \ytt_1 \cdots \ytt_n.$$

\begin{proof}[Proof of Theorem \ref{thm: orbifold rel gap} for nonorientable orbifolds]
	The proof is similar to the orientable case. It again suffices to show that $\scl_{(B,\partial B)}(g)\ge 1/24$ assuming that we are not in cases (\ref{item: vanish if torsion}), (\ref{item: vanish if conjugate to inv}) or (\ref{item: vanish on the boundary}).
	
	Identify $G=\pi_1(B)$ with the free product of cyclic groups above. Now the $m+1$ boundary elements (as homology classes) are linearly independent in $H_1(\pi_1(B);\R)$ and span an $(m+1)$-dimensional subspace $V_{\partial B}$. We assume that $[g]$ lies in $V_{\partial B}$. Then there is a unique chain $c=c_0 b+\sum c_i \btt_i$ for some $c_i\in\frac{1}{2}\Z$ such that $[g]=[c]$, and $\scl_{(B,\partial B)}(g)=\scl_B(g-c)$, where
	$$
	b \defeq \xtt_1^2 \cdots \xtt_k^2 \btt_1 \cdots \btt_m \ytt_1 \cdots \ytt_n.
	$$
	As $[g]\in V_{\partial B}$, the assumptions imply that $g$ is not conjugate into any free factor unless $k=1$ and $g$ is conjugate to $\xtt_1^\ell$ for some $\ell\neq0\in \Z$. In this exceptional case, $c=\frac{\ell}{2} (b-\sum \btt_i)$, and $|\bar{\phi}_b(g-c)|=|\ell|/2$, where $\bar{\phi}_b$ is the Brooks quasimorphism of defect at most $6$ by Lemma \ref{lemma: Brooks} since $b$ is not self-overlapping and has length at least $2$. Hence Bavard's duality implies that $\scl_B(g-c)\ge |\ell|/24\ge 1/24$ in this case.
	
	In the sequel, assume that $g$ is not conjugate into any free factor. Write $g = h^q$ with $q\in\Z_+$ so that $h$ is not a proper power. Let $h'$ be the conjugate of $h$ as in Lemma \ref{lemma: cyclic conjugate minimal}. Note that $h'$ and $h'^{-1}$ are not conjugate since $g$ and $g^{-1}$ are not, and that $h'$ has length at least $2$ since $g$ is not conjugate into any free factor. Thus $\bar{\phi}_{h'}(g)=q \geq 1$ and $\bar{\phi}_{h'}(\btt_i)=0$ for all $i$ by Lemma \ref{lemma: Brooks}, and $\bar{\phi}_{h'}$ has defect at most $6$.
	
	If $\bar{\phi}_{h'}(b) = 0$ then we compute by Bavard's duality to obtain
	$$
	\scl_{(B,\partial B)}(g)=\scl_B(g -c) \geq \frac{|\bar{\phi}_{h'}(g -c)|}{2 D(\bar{\phi}_{h'})} \geq \frac{|\bar{\phi}_{h'}(g)|}{12} \geq \frac{1}{12}.
	$$
	
	We are left with the case where $\bar{\phi}_{h'}(b) \neq 0$. By Lemma \ref{lemma: Brooks} and the assumption that $g$ is not boundary parallel, this implies $|b|>|h'|\ge 2$ and moreover either $h'$ or its inverse is cyclically a proper subword of $b$. Suppose without loss of generality that $h'$ (instead of $h'^{-1}$) appears in $b$ up to cyclic permutation.
	
	From the expression of $b$ we see that it only contains one copy of $h'$ since $|h'|>1$, and it does not contain $h'^{-1}$. 
	This is because the exception of $h'=\btt_i \btt_{i+1}$ with $o_i=o_{i+1}=2$ cannot occur as $h'$ is not conjugate to $(h')^{-1}$.
	Thus we have $\bar{\phi}_{h'}(b)=1$. Hence $\bar{\phi}_{h'}(g-c)=q-c_0\in\frac{1}{2}\Z$ and $\scl_B(g-c)\ge |q-c_0|/12$ by Bavard's duality. It implies $\scl_B(g-c)\ge1/24$ unless $c_0=q$.
	For this exceptional case, observe that $\bar{\phi}_{b}$ has defect at most $6$ and $\bar{\phi}_{b}(h')=0=\bar{\phi}_{b}(\btt_i)$ by Lemma \ref{lemma: Brooks}, and hence $\bar{\phi}_{b}(g)=0, \bar{\phi}_{b}(c)=c_0$.
	Set $\bar{\phi} \defeq \bar{\phi}_{h'} - \bar{\phi}_{b}$, which has defect $D(\bar{\phi})\leq  D(\bar{\phi}_{h'})+ D(\bar{\phi}_{b}) \leq 12$. Then $\bar{\phi}(c) = c_0-c_0=0$ and
	$\bar{\phi}(g) = q\bar{\phi}_{h'}(h')-q\bar{\phi}_{b}(h')=q \geq 1$.
	Thus
	$$
	\scl_B(g -c) \geq \frac{\bar{\phi}(g -c)}{2 D(\bar{\phi})} \geq \frac{1}{24},
	$$
	which completes the proof.
\end{proof}

\section{Uniform relative spectral gap of $2$-orbifolds}\label{sec: closed case}
The goal of this section is to prove Theorem~\ref{thmA: orbifold uniform gap}.

Given Theorem \ref{thm: orbifold rel gap}, it suffices to focus on the case of closed orbifolds.
A different proof method is required when the orbifold is the $2$-sphere with three cone points. We first handle the generic case in Section \ref{subsec: large orbifolds} and then consider the exceptional case of a sphere with three cone points in Section \ref{subsec: small orbifolds}.

\subsection{The generic case}\label{subsec: large orbifolds}
\begin{thm}\label{thm: closed orbifold generic case}
	Let $B$ be a closed $2$-dimensional orbifold $B$ with $\chi_o(B)<0$. Suppose $B$ is not the $2$-sphere with three cone points.
	Then for all $g\in\pi_1(B)$ we have either $\scl_{B}(g)\ge 1/36$ or $\scl_{B}(g)=0$.
\end{thm}

Our proof relies on a result of Clay--Forester--Louwsma \cite{CFL16}, so we start by recalling some relevant definitions.
\begin{defn}\label{def: acylindrical}
	A group action on a tree (without inversion) simplicially is \emph{$K$-acylindrical} for some $K\in\Z_+$ if the fixed point set of each nontrivial element has diameter at most $K$.
\end{defn}
It is well known that any isometry of a tree is either elliptic (i.e. acting with fixed points) or hyperbolic (i.e. acting by translation along a unique axis).
In addition, the fixed point set of an elliptic isometry is a subtree.

With the definition above, which differs from the one in \cite{CFL16}, 
the statement below is equivalent to \cite[Theorem 6.11]{CFL16}.
\begin{thm}[Clay--Forester--Louwsma \cite{CFL16}]\label{thm: acyl action gap}
	Let $G$ be a group acting $K$-acylindrically and let $N$ be the least integer greater than or equal to $(K+3)/2$. 
	If $g\in G$ acts without fixed point (i.e. by hyperbolic isometry), then either $\scl_G(g)\ge 1/(12N)$ or $\scl_G(g)=0$.
	Moreover, the latter case occurs exactly when $g$ and $g^{-1}$ are conjugate.
\end{thm}

For our applications below, we consider the Bass--Serre tree actions for HNN extensions and amalgamated free products \cite{Serre}.
In such cases, an element acts elliptically if and only if it is conjugate into a vertex group.
For the HNN extension $G$ associated to an isomorphism of two subgroups $C_1,C_2\le A$,
an element acting on the Bass--Serre tree with a fixed point set of diameter $\ge2$ must fix two adjacent edges and hence is conjugate into $C_i\cap hC_jh^{-1}$ for some $i,j\in\{1,2\}$ and $h\in A\setminus C_j$.
Thus the action is $1$-acylindrical if all such subgroups are trivial.
For an amalgamated free product $G=A\star _C B$ acting on the Bass--Serre tree, 
if the fixed point set of an element $g\in G$ has diameter $\ge3$, then it contains a path of length $3$.
Up to conjugation, we may assume the middle edge of such a path is the one fixed by $C$, in which case $g$ lies in $aCa^{-1}\cap C$ for some $a\in A\setminus C$ and also in $bCb^{-1}\cap C$ for some $b\in B\setminus C$. In particular, the action is $2$-acylindrical if $aCa^{-1}\cap C=1$ for all $a\in A\setminus C$ or $C\cap bCb^{-1}=1$ for all $b\in B\setminus C$.

Now we prove Theorem \ref{thm: closed orbifold generic case}.
\begin{proof}[Proof of Theorem \ref{thm: closed orbifold generic case}]
	
	First consider the orientable case. Suppose $B$ is orientable with genus $k$ and $n$ cone points of orders $o_1,\cdots, o_n$. Then $\chi_o(B)=2-2k-\sum (1-1/o_i)$ and 
	$$\pi_1(B)=\langle \xtt_1,\ldots,\xtt_{2k},\ytt_1,\dots \ytt_n \ |\ \ytt_i^{o_i}=1, [\xtt_1,\xtt_2]\ldots[\xtt_{2k-1},\xtt_{2k}]\ytt_1\ldots\ytt_n=1 \rangle.$$
	If $k\ge 1$, we rewrite the last relation as
	$\xtt_2^{-1}[\xtt_3,\xtt_4]\ldots[\xtt_{2k-1},\xtt_{2k}]\ytt_1\ldots\ytt_n=\xtt_1\xtt_2^{-1}\xtt_1^{-1}$, and view
	$\pi_1(B)$ as the HNN extension over $\Z$ of the free product $H$ generated by $\xtt_2,\ldots,\xtt_{2k},\ytt_1,\ldots, \ytt_n$,
	where the generator of $\Z$ is sent to $\xtt_2^{-1}$ and $\xtt_2^{-1}[\xtt_3,\xtt_4]\ldots[\xtt_{2k-1},\xtt_{2k}]\ytt_1\ldots\ytt_n$ respectively. 
	As $\chi_o(B)<0$ and $k\ge1$, we have either $k\ge2$ or $n\ge1$.
	By looking at the projection to
	either $\langle\xtt_3,\xtt_4\rangle$ (if $k\ge2$) or $\langle \ytt_1\rangle$ (if $n\ge1$), the cyclic subgroups of $H$ generated by 
	$\xtt_2^{-1}$ and $\xtt_2^{-1}[\xtt_3,\xtt_4]\ldots[\xtt_{2k-1},\xtt_{2k}]\ytt_1\ldots\ytt_n$ respectively have no conjugates intersecting non-trivially. Thus the $\pi_1(B)$ action on the Bass--Serre tree associated to this HNN extension is $1$-acylindrical, and we have a gap $1/24$ for hyperbolic elements by Theorem \ref{thm: acyl action gap}. 
	
	If $k=0$, then we have $n\ge 4$ since $\chi_o(B)<0$ and $B$ is not the $2$-sphere with three cone points. 
	Then $\pi_1(B)$ can be viewed as an amalgam over $\Z$ of the free products generated by $\ytt_1,\ytt_2$ and $\ytt_3,\ldots,\ytt_n$ respectively, where the generator of $\Z$ is sent to $(\ytt_1 \ytt_2)^{-1}$ and $\ytt_3\cdots\ytt_n$ respectively. Note that $h(\ytt_1 \ytt_2)^p h^{-1}=(\ytt_1 \ytt_2)^q$ with $p,q\neq0$ only when $p=q$ and $h$ is a power of $\ytt_1 \ytt_2$, unless $o_1=o_2=2$. In the exceptional case of $o_1=o_2=2$, either $n>4$ or $o_3,o_4$ are not both equal to $2$ since $\chi_o(B)<0$, either of which implies that $h(\ytt_3\cdots\ytt_n)^p h^{-1}=(\ytt_3\cdots\ytt_n)^q$ with $p,q\neq0$ only when $p=q$ and $h$ is a power of $\ytt_3\cdots\ytt_n$. It follows that the $\pi_1(B)$ action on the Bass--Serre tree is $2$-acylindrical, and we have a gap $1/36$ for hyperbolic elements by Theorem \ref{thm: acyl action gap}. 
	
	Now suppose $B$ is nonorientable with genus $k$ and $n$ cone points of orders $o_1,\cdots, o_n$, where $k\ge1$. Then $\chi_o(B)=2-k-\sum (1-1/o_i)$ and 
	$$\pi_1(B)=\langle \xtt_1,\ldots,\xtt_k,\ytt_1,\dots \ytt_n \ |\ \ytt_i^{o_i}=1, \xtt_2^2\ldots\xtt_k^2\ytt_1\ldots\ytt_n=\xtt_1^{-2} \rangle.$$
	The last relation naturally splits $\pi_1(B)$ as an amalgam over $\Z$ of $\langle\xtt_1\rangle$ and the free product of $\xtt_2,\ldots, \xtt_k, \ytt_1,\ldots,\ytt_n$. 
	Since $\chi_o(B)<0$, we have $k\ge2$ or $n\ge3$, so $h(\xtt_2^2\ldots\xtt_k^2\ytt_1\ldots\ytt_n)^p h^{-1}= (\xtt_2^2\ldots\xtt_k^2\ytt_1\ldots\ytt_n)^q$ for $p,q\neq0$ only when $p=q$ and $h$ is a power of $\xtt_2^2\ldots\xtt_k^2\ytt_1\ldots\ytt_n$. Thus the $\pi_1(B)$ action on the Bass--Serre tree is $2$-acylindrical, and we have a gap $1/36$ for hyperbolic elements by Theorem \ref{thm: acyl action gap}.
	
	It remains to show the gap for elliptic elements. In any of the cases above, if $g$ is an elliptic element, we have $\scl_B(g)\ge \scl_{(B',\partial B')}(g)$ by \cite[Lemma~3.2]{CH:sclgapgog_new}, where $B'$ is the sub-orbifold corresponding to the vertex group supporting $g$ in the splittings above. Theorem \ref{thm: orbifold rel gap} implies that $\scl_{(B',\partial B')}(g)\ge1/24$ unless one of the three exceptional cases occurs. For the exceptional cases (\ref{item: vanish if torsion}) and (\ref{item: vanish if conjugate to inv}), we have $\scl_B(g)=0$. The remaining case (\ref{item: vanish on the boundary}) implies that $g$ lies in the edge group $\Z=\langle z\rangle$. 
	\begin{enumerate}
		\item In the HNN extension for $B$ orientable with genus $k\ge1$ above, nontrivial elements in the edge group have nontrivial image in the first homology and thus have $\scl_B(g)=\infty$.
		\item In the amalgam for $B$ orientable with genus $k=0$ above, let $G_1,G_2$ be the vertex groups. By \cite[Theorem~4.8]{CH:sclgapgog_new}, we have $\scl_B(z)=\min(\scl_{G_1}(\ytt_1 \ytt_2),\scl_{G_2}(\ytt_3\cdots\ytt_n))$, which is either $0$ or at least $1/12$ by Corollary~\ref{cor: free product gap for elements} since $G_1$ and $G_2$ are free products of cyclic groups.
		\item In the amalgam for $B$ nonorientable above, nontrivial elements in the edge group are nontrivial in the homology and $\scl_B(g)=\infty$, unless $k=1$. If $k=1$, then \cite[Theorem~4.8]{CH:sclgapgog_new} implies $\scl_B(z)=\scl_{H}(\ytt_1\ldots\ytt_n)\ge 1/12$ by Corollary~\ref{cor: free product gap for elements}, where $H$ is the free product of cyclic groups generated by $\ytt_i$'s; it is easy to check that $\ytt_1\ldots\ytt_n$ cannot be conjugate to its inverse as $n\ge3$.
	\end{enumerate}
\end{proof}

\subsection{The exceptional case}\label{subsec: small orbifolds}
In the exceptional case where $B$ is the $2$-sphere with three cone points of order $p,q,r$, we have $1/p+1/q+1/r<1$ when $\chi_o(B)<0$. In this case the orbifold fundamental group is called the (hyperbolic) von Dyck group, which is the index two subgroup of the (hyperbolic) triangle group with parameters $p,q,r$. The tiling of the hyperbolic plane by hyperbolic triangles of angles $\pi/p,\pi/q,\pi/r$ gives a faithful representation of the triangle group and thus of $G$.
\begin{lemma}\label{lemma: length estimate}
	There is a uniform constant $\delta\defeq 2\cosh^{-1}(\cos(2\pi/7)+1/2)$ such that any hyperbolic element in a von Dyck group $G$ has translation length at least $\delta$.
\end{lemma}
\begin{proof}
	Any hyperbolic element $\gamma\in G$ is hyperbolic in the corresponding triangle group. Thus by \cite[Proposition 3.1]{length_gap}, the trace of $\gamma$ as an element of $\mathrm{PSL}_2 \R$ satisfies 
	$$|\mathrm{tr}(\gamma)|\ge 2\cos(2\pi/7)+1=2\cosh(\delta/2).$$
	Hence the translation length of $\gamma$ is at least $\delta$, and our conclusion follows.
\end{proof}

\begin{thm}\label{thm: von Dyck gap}
	Let $G$ be a hyperbolic von Dyck group. There is a uniform constant $C>0$ independent of $G$ such that either $\scl_G(\gamma)\ge C$ or $\scl_G(\gamma)=0$ for any $\gamma\in G$, and the latter case occurs if and only if $\gamma$ has finite order or is conjugate to its inverse.
    In particular, the exceptional case occurs only when
    $\gamma^n$ is conjugate to $\gamma^{-n}$ for some $n\in\Z_+$.
    
\end{thm}
	
Our proof uses a modification of the proof of \cite[Theorem C]{Cal:sclhypmfd}; see also \cite[Chapter 3]{Cal:sclbook}. 
For the ease of future references, we formulate a key part of the argument in a more general setting as Lemma \ref{lemma: Calegari lemma} below.
For this general setup, we consider a \emph{pleated surface} $f:S\to M$ in a complete hyperbolic manifold $M$. Such a surface arises as a good representative in the homotopy class of a relative admissible surface (Proposition~\ref{prop: relative adm surf}).
Here a pleated surface $(S,f)$ consists of two pieces of data. The underlying surface $S$ is a hyperbolic surface of finite volume with geodesic boundary. 
The map $f:S\to M$ takes cusps into cusps and preserves the lengths of all rectifiable curves, 
and each point $p\in S$ is in the interior of a straight line segment which is mapped by $f$ to a straight line segment. 
The key properties that we will use are $\area(S)=-2\pi\chi(S)$ and that $f$ preserves lengths.

We are actually interested in the even more general setup where $M$ is an \emph{orbifold} with orbifold fundamental group $G$, which we set to be the von Dyck group in the proof of Theorem \ref{thm: von Dyck gap}. To avoid technical terms, we will just take a finite-index torsion-free subgroup $\Gamma$ of $G$ and take finite covers to reduce the problem to pleated surfaces in \emph{manifolds}. In the statement below, the group $G$ is not assumed to be a finite extension of $\Gamma$, although it is the case of main interest.

\begin{lemma}[Calegari \cite{Cal:sclhypmfd}]\label{lemma: Calegari lemma}
	Fix a constant $\kappa>0$. Let $G$ be a subgroup of $\Isom^+(\Hbb^n)$ and let $\Gamma$ be a torsion-free discrete subgroup of $G$.
	Consider an arbitrary hyperbolic element $\gamma\in G$ of translation length at least $\kappa$ so that some nontrivial power of $\gamma$ lies in $\Gamma$.
	Let $M$ be the complete hyperbolic manifold $\Hbb^n/\Gamma$ and let $S$ be a surface of finite type. 
	Suppose $f: S\to M$ is a pleated surface (described as above) that takes cusps to cusps and each geodesic boundary to a geodesic loop in $M$ representing an element in $\Gamma$ that is conjugate to a positive power of $\gamma$ in $G$. 
	Then there is a constant $\mathcal{E}=\mathcal{E}(\kappa)>0$ 
	such that for any $0<\epsilon\le \mathcal{E}$ there is a segment $\sigma$ on $\partial S$ satisfying
    \begin{equation}\label{eqn: length lower bound}
        \len(\sigma)\ge \frac{\len(\partial S)}{-6\chi(S)}-\frac{2\pi}{3\epsilon};
    \end{equation}
	and we also have
    \begin{equation}\label{eqn: length upper bound}
        \len(\sigma)\le 2\cdot\len(\gamma)+4\epsilon
    \end{equation}
	unless there are $g,g'\in G$ conjugate to $\gamma$, and a common $m\in \Hbb^n$ such that $d(m,gg' m)<4\epsilon$ and $d(m,g'g m)<4\epsilon$.
\end{lemma}
\begin{proof}
	We essentially follow the same proof strategy as in \cite[Lemma 3.7]{Cal:sclbook}. 
	Consider the double $DS$ of $S$ across the geodesic boundaries. Take the $2\epsilon$ thin-thick decomposition of $DS$ into $DS_{thin}$ and $DS_{thick}$ with respect to some $\epsilon>0$ small compared to the $2$-dimensional Margulis constant. 
	Denote $S_{thin}\defeq S\cap DS_{thin}$ and $S_{thick}\defeq S\cap DS_{thick}$ respectively. If $\epsilon$ is small compared to $\len(\partial S)$, then $S_{thin}$ consists of cusp neighborhoods and $r$ rectangles, where each rectangle doubles to an annulus in $DS_{thin}$. Such annuli are disjoint in $DS$, hence there are at most $-3\chi(DS)/2=-3\chi(S)$ these annuli, i.e. $r\le -3\chi(S)$.
	Note that this setup works for any $\epsilon$ smaller than a constant $\mathcal{E}$ determined by $\kappa$ as $\len(\partial S)\ge\len(\gamma)\ge\kappa$ by our assumption.
	
	Note that each rectangle in $S_{thin}$ has two opposite sides on $\partial S$ at Hausdorff distance no more than $\epsilon$. So the $r$ disjoint rectangles give rise to $2r$ disjoint segments in $S_{thin}\cap \partial S$, and it follows from the pigeonhole principle that there is at least one such segment $\sigma$ satisfying
	$$\len(\sigma)\ge \frac{\len(S_{thin}\cap\partial S)}{2r}\ge \frac{\len(S_{thin}\cap\partial S)}{-6\chi(S)}.$$
	To estimate $\len(S_{thin}\cap\partial S)$, note that $\len(\partial S)=\len(S_{thin}\cap\partial S)+\len(S_{thick}\cap\partial S)$. In addition, the $\epsilon/2$-neighborhood of $S_{thick}\cap\partial S$ is embedded in $S$, so we have $$\frac{\epsilon}{2}\cdot \len(S_{thick}\cap\partial S) \le \area (S_{thick}\cap\partial S)\le\area(S)=-2\pi\chi(S).$$
	Combining such estimates, we obtain the first desired inequality:
	$$\len(\sigma)\ge \frac{\len(S_{thin}\cap\partial S)}{-6\chi(S)}\ge \frac{\len(\partial S)}{-6\chi(S)}-\frac{2\pi}{3\epsilon}.$$
	
	To obtain the other inequality, let $R$ be the rectangle in $S_{thin}$ containing $\sigma$.
	The restriction of the map $f:S\to M$ to $R$ lifts to the universal cover $\widetilde{M}=\Hbb^n$, on which $G$ acts by isometry. Denote the image of $R$ in $\Hbb^n$ as $\widetilde{R}$.
	Then one side $\tilde{\sigma}$ of $\widetilde{R}$ projects to $\sigma$, and we denote its opposite side as $\tilde{\sigma}'$. 
	Note that $\tilde{\sigma}$ and $\tilde{\sigma}'$ have Hausdorff distance no more than $\epsilon$ as $R$ is in the $\epsilon$-thin part and the map $f$ is length-preserving.
	By our assumption, $\tilde{\sigma}$ (resp. $\tilde{\sigma}'$) lies on the axis $L$ (resp. $L'$) of some element $g\in G$ (resp. $g'\in G$) conjugate to $\gamma$. Orient the axes in a way consistent with the orientation of $\tilde{\sigma}$ and $\tilde{\sigma'}$ respectively induced from the orientation on $\partial S$. Then $g$ and $g'$ act on the axes by translation in the positive direction respectively, and the two axes $L$ and $L'$ are almost anti-aligned; see Figure \ref{fig: sigmatilde}.
	
	\begin{figure}
		\labellist
		\small 
		\pinlabel $p$ at 30 -5
		\pinlabel $s$ at 60 0
		\pinlabel $m$ at 202 12
		\pinlabel $\tilde{\sigma}$ at 155 0
		\pinlabel $q$ at 375 -5
		\pinlabel $r$ at 347 0
		\pinlabel $L$ at 412 0
		
		\pinlabel $p'$ at 47 85
		\pinlabel $s'$ at 62 84
		\pinlabel $m'$ at 205 70
		\pinlabel $\tilde{\sigma}'$ at 255 83
		\pinlabel $q'$ at 365 85
		\pinlabel $r'$ at 344 83
		\pinlabel $L'$ at 415 85
		\endlabellist
		\centering
		\includegraphics[scale=0.7]{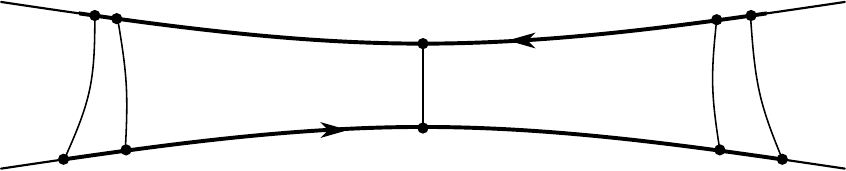}
		\caption{The geodesic segment $\tilde{\sigma}$ on the axis $L$ and the corresponding segment $\tilde{\sigma}'$ on the axis $L'$. 
			Corresponding points, say $r$ and $r'$, on the two segments are at distances no more than $\epsilon$.}\label{fig: sigmatilde}
	\end{figure}
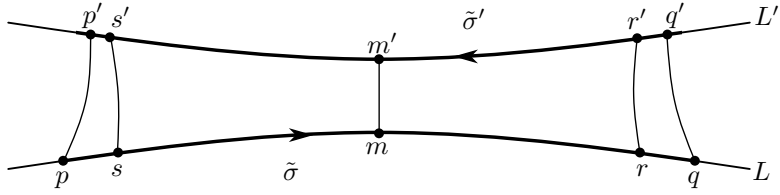
	
	Suppose to contrary that $\len(\tilde{\sigma})=\len(\sigma)>2\cdot\len(\gamma)+4\epsilon$. Let $p,q$ be the endpoints of $\tilde{\sigma}$, and let $m$ be the midpoint of $\tilde{\sigma}$. 
	We use $d_L^y(x)$ to denote the \emph{signed} distance of a point $x\in L$ to $y\in L$ so that it is positive if and only if $x$ is in the positive direction of $y$ with respect to the orientation on $L$. We use a similar notation $d_{L'}^y(x)$ for signed distance on $L'$.
	Then up to swapping $p$ and $q$, we have $d_L^m(p)=-\len(\sigma)/2 <-\len(\gamma)-2\epsilon$ and $d_L^m(q)=\len(\sigma)/2 >\len(\gamma)+2\epsilon$.
	Since $\tilde{\sigma}$ and $\tilde{\sigma}'$ have Hausdorff distance no more than $\epsilon$, there are points $p',q',m'$ on $\tilde{\sigma}'$ such that $d(p',p)\le \epsilon$ and similarly for $q'$ and $m'$. By the triangle inequality, we have $d(p',m')\ge d(p,m)-d(p,p')-d(m,m')> \len(\gamma)$. Using the fact that $L$ and $L'$ are almost anti-aligned, we have $d_{L'}^{m'}(p')>\len(\gamma)$, and by a similar argument we have $d_{L'}^{m'}(q')<-\len(\gamma)$.
	
	The isometry $g$ takes $m$ to some point $r\defeq gm$ on $L$.
	As $d(r,m)$ is the translation length of $g$ and $g$ is conjugate to $\gamma$, we have $d_L^m(r)=\len(\gamma)<d_L^m(q)$, so $r$ lies on $\tilde{\sigma}$.
	Thus there is a point $r'$ on $\tilde{\sigma}'$ with $d(r',r)\le\epsilon$. 
	By the triangle inequality, we have 
	$$d(m',r')\le d(m',m)+d(m,r)+d(r,r')\le \len(\gamma)+2\epsilon,$$
	and similarly $d(m',r')\ge \len(\gamma)-2\epsilon$.
	Incorporating the orientation, we have $|d_{L'}^{m'}(r')+\len(\gamma)|\le 2\epsilon$.
	This means $d(g'^{-1}m',r')\le 2\epsilon$, or equivalently $d(m',g'r')\le 2\epsilon$.
	Now by the triangle inequality, we obtain
	$$d(m,g'gm)\le d(m,m')+d(m',g'r')+d(g'r',g'gm)\le \epsilon+2\epsilon + d(r',gm)=3\epsilon + d(r',r)\le 4\epsilon.$$
	
	Now we deduce $d(m,gg'm)\le 4\epsilon$ analogously. Let $s'\defeq g'm'$. We showed earlier $d_{L'}^{m'}(p')>\len(\gamma)$, so $s'$ lies on $\tilde{\sigma}'$, and there is a point $s$ on $\tilde{\sigma}$ with $d(s,s')\le \epsilon$. Similar to the argument in the paragraph above, we have $|d_L^m(s)+\len(\gamma)|\le 2\epsilon$. 
	This means that $d(m,gs)\le 2\epsilon$.
	So by the triangle inequality, we have
	$$d(m,gg'm)\le d(m,gs)+d(gs,gg'm)\le 2\epsilon+ d(s,g'm)\le 2\epsilon + d(s,s') +d(s',g'm)\le 3\epsilon+d(m',m)\le 4\epsilon.$$
\end{proof}

Now we apply Lemma~\ref{lemma: Calegari lemma} to prove Theorem~\ref{thm: von Dyck gap}.
\begin{proof}
    It suffices to give a uniform lower bound $\scl_G(\gamma)\ge C$ for all hyperbolic elements $\gamma\in G$ for which $\gamma$ is not conjugate to its inverse. In fact, given this, $\scl_G(\gamma)=0$ implies that $\gamma$ is either elliptic, which has finite order, or is conjugate to its inverse. In both situations, some power $\gamma^n$ is conjugate to its inverse for some $n\in\mathbb{Z}_+$, which in turn implies $\scl_G(\gamma)=0$.

	So from now on, we consider a hyperbolic element $\gamma\in G$ for which $\gamma$ is not conjugate to its inverse.
    
    Fix a torsion-free finite index normal subgroup $\Gamma$ of $G$, and let $d$ be the size of the finite quotient $H=G/\Gamma$. We think of $\Gamma$ as the fundamental group of a hyperbolic closed surface $\Sigma$, where the hyperbolic structure is induced from the embedding of $G$ in $\mathrm{PSL}_2\R$.
	
	Then $\gamma^k\in \Gamma$ for some $1\le k\le d$. By the index formula (Lemma~\ref{lemma: index formula}),
	we have
	$$\scl_G(\gamma^k)=\frac{1}{d} \scl_\Gamma(\sum_{h\in H} h\gamma^k h^{-1})$$
	It suffices to give a lower bound of $\frac{1}{kd}\scl_\Gamma(\sum_h h\gamma^k h^{-1})$ that is independent of $\gamma\in G$ (and $k$) for which $\gamma^n$ is not conjugate to $\gamma^{-n}$ for all $n\in\Z_+$.
    
	Consider any admissible surface $f:S\to \Sigma$ for the chain $c\defeq \sum_{h\in H} h\gamma^k h^{-1}$ of a certain degree $n$.
    By Theorem~\ref{thm: restrict to pleated} (with $M=\Sigma$, which has no cusp), for the sake of estimating $\scl_\Gamma(c)$, we may assume that $f$ has a pleated representative. Now we apply Lemma~\ref{lemma: Calegari lemma} with the constant $\kappa=\delta$ from Lemma~\ref{lemma: length estimate}. 

    \begin{claim}\label{claim: rule out exceptional}
        There is a uniform constant $\mathcal{E}'>0$ such that for any $0<\epsilon\le\mathcal{E}'$ and any hyperbolic element $\gamma$ in a von Dyck group $G$ for which $\gamma$ is not conjugate to its inverse, there are no conjugates $g,g'\in G$ of $\gamma$ and $m\in \Hbb^2$ with $d(m,gg' m)<4\epsilon$ and $d(m,g'g m)<4\epsilon$ for the standard action of $G$ on $\Hbb^2$.
    \end{claim}
    
    Given the claim above, choose $\epsilon=\mathcal{E}'$, then
    we obtain a segment $\sigma$ on $\partial S$,
    satisfying both inequalities (\ref{eqn: length lower bound}) and (\ref{eqn: length upper bound}).

    Note that each $h\gamma^kh^{-1}$ has length $k\cdot\len(\gamma)$ and there are $d$ choices of $h\in H$, so
    $\len(\partial S)=nkd\cdot \len(\gamma)$ since $S$ has degree $n$.
    
    The inequalities (\ref{eqn: length lower bound}) and (\ref{eqn: length upper bound}) together imply
    $$\left(\frac{nkd}{-6\chi^-(S)}-2\right)\len(\gamma)\le 4\epsilon+\frac{2\pi}{3\epsilon}.$$
    Using $\len(\gamma)\ge\delta$, we get
    $$\frac{nkd}{-6\chi^-(S)}-2\le\frac{1}{\delta}\left(4\epsilon+\frac{2\pi}{3\epsilon}\right),$$
    or equivalently,
    $$\frac{nkd}{-6\chi^-(S)}\le2+\frac{1}{\delta}\left(4\epsilon+\frac{2\pi}{3\epsilon}\right).$$
    This yields a positive lower bound of $-\chi^-(S)/nkd$ and hence a bound of $\frac{1}{kd}\scl_\Gamma(c)$. Moreover, the bound is uniform since $\epsilon$ is.
    Given this, we obtain a uniform lower bound $C$ of $\scl_G(\gamma)$ for any hyperbolic element $\gamma$ that is not conjugate to its inverse.

    So it remains to prove the claim.
    \begin{proof}[Proof of Claim~\ref{claim: rule out exceptional}]
    	We choose $\epsilon$ such that $8\epsilon$ is less than the constant $\delta$ from Lemma~\ref{lemma: length estimate} and
    	$$\sinh (2\epsilon)\le \frac{1}{\sqrt{2}}\sinh\frac{\delta}{2}.$$
    	
        Suppose the exceptional situation occurs, namely, there are $g,g'\in G$ conjugate to $\gamma$, and a common $m\in \Hbb^2$ such that $d(m,gg' m)<4\epsilon$ and $d(m,g'g m)<4\epsilon$. Then Lemma~\ref{lemma: length estimate} implies that $gg'$ and $g'g$ have translation length smaller than $4\epsilon<\delta$, so both must be elliptic. 
	
    Now suppose $gg'$ is elliptic fixing some $z\in\Hbb^2$, then $g'g$ is also elliptic with the same angle of rotation fixing $g'(z)$, which is distinct from $z$ since $g'$ is hyperbolic. 
    For an elliptic element $h\in \PSLtwo(\R)$ fixing a point $z$ with an angle of rotation $\theta\in [-\pi,\pi]$, for any $t>0$,
    the set of points $m$ with $d(hm,m)< t$ is the disk of some radius $r$ centered at $z$, and we have
    $$\sinh \frac{t}{2}=\sinh r \cdot \sin \frac{\theta}{2}.$$
    Now since $z$ and $g'(z)$ are at least $\delta$ apart, the existence of the common $m$ with $t=4\epsilon$ implies that $r\ge \delta/2$, so 
    $$\sin\frac{\theta}{2}\le \frac{\sinh(2\epsilon)}{\sinh(\delta/2)}.$$
    By our choice of $\epsilon$, we know $gg'$ and $g'g$ rotate by a common angle $|\theta|<\pi/2$.
	In particular, one can easily observe that $(gg')(g'g)^{-1}=[g,g']\in G$ has a fixed point on $\partial \Hbb^2$ (see Figure \ref{fig: hyperbolic}), 
	which implies either $[g,g']=id$ or $[g,g']$ is a hyperbolic element. 
	The former case cannot happen since $g$ and $g'$ have distinct fixed points. 
	In the latter case $[g,g']$ has translation length no more than $8\epsilon<\delta$
	as it takes the point $p_1$ to $p_3$ in Figure \ref{fig: hyperbolic} with 
	$$d(p_1,p_3)\le d(p_1,p_2)+d(p_2,p_3)\le 4\epsilon+4\epsilon=8\epsilon.$$
	This is impossible since translation lengths of hyperbolic elements in von Dyck groups are at least $\delta$ by Lemma~\ref{lemma: length estimate}.
    \end{proof}
    
    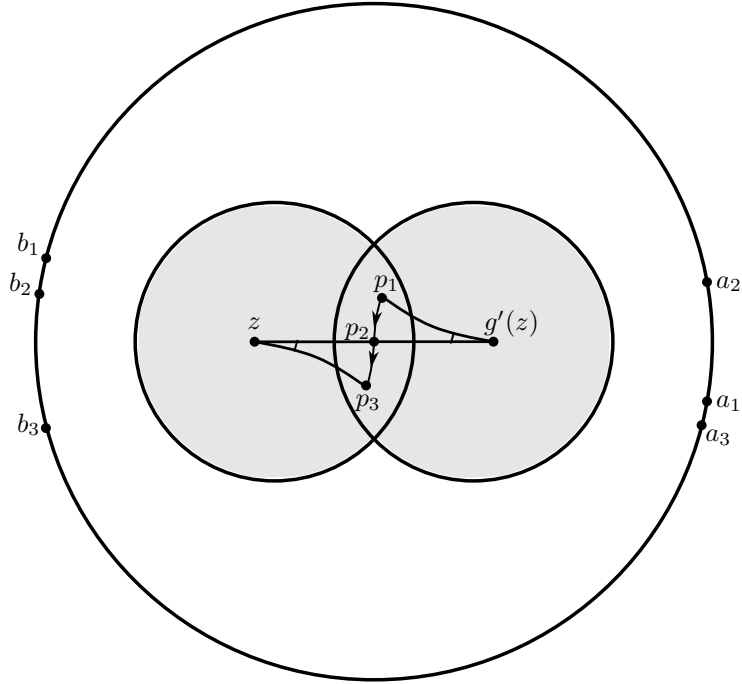
\begin{figure}
    	\labellist
    	\small 
    	\pinlabel $z$ at 83 135
    	\pinlabel $g'(z)$ at 180 135
    	\pinlabel $p_1$ at 133 150
    	\pinlabel $p_2$ at 122 132
    	\pinlabel $p_3$ at 126 105
    	\pinlabel $a_1$ at 262 105
    	\pinlabel $a_2$ at 262 150
    	\pinlabel $a_3$ at 258 92
    	\pinlabel $b_1$ at -2 165
    	\pinlabel $b_2$ at -5 150
    	\pinlabel $b_3$ at -2 97
    	\endlabellist
    	\centering
    	\includegraphics[scale=1]{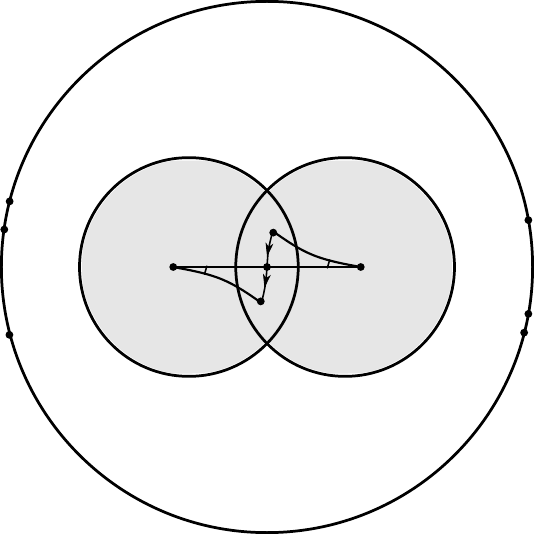}
    	\caption{The elliptic isometry $gg'$ fixes $z$ and rotates by an angle $\theta<\pi/2$ clockwise, and $(g'g)^{-1}$ fixes $g'(z)$ and rotates by the same angle $\theta$ counterclockwise. As $(g'g)^{-1}$ moves points on the right (resp. left) slower (faster) than $gg'$ in Euclidean distance, we can find $a_2=(g'g)^{-1} a_1$ and $a_3=gg'(a_2)$ with relative positions as shown, and similarly $b_2=(g'g)^{-1} b_1$ and $b_3=gg'(b_2)$. This implies that the lower circular arc from $a_1$ to $b_1$ is mapped inside itself by $[g,g']=gg'(g'g)^{-1}$, showing the existence of a fixed point.
    	The shaded disks are points moved by at most $4\epsilon$ under $g'g$ and $gg'$ respectively, so $d(p_1,p_2)\le 4\epsilon$ and $d(p_2,p_3)\le 4\epsilon$.}\label{fig: hyperbolic}
    \end{figure}
\end{proof}

Finally, we deduce Theorem~\ref{thmA: orbifold uniform gap}.
\begin{proof}[Proof of Theorem~\ref{thmA: orbifold uniform gap}]
	If $B$ has nonempty boundary, then the result follows from Theorem~\ref{thmA: orbifold uniform gap with boundary} with a better bound $1/24$.
	If $B$ is closed and is not a sphere with $3$ cone points, then it follows from Theorem~\ref{thm: closed orbifold generic case} with an explicit bound $1/36$.
	Finally, if $B$ is a sphere with $3$ cone points, then $G=\pi_1(B)$ is a hyperbolic von Dyck group, so the result follows from Theorem~\ref{thm: von Dyck gap}.
\end{proof}

\appendix

\section{Existence of pleated representative}

The main goal of the appendix is to show the following theorem, which allows us to restrict our attention to pleated surfaces when we compute scl in finite-volume hyperbolic manifolds relative to cusps (i.e. the space $\Ends(M)$ of ends). The proof methods provide the existence of pleated surfaces in a general setting, which might be well-known to experts but seems hard to find in the literature.
\begin{thm}\label{thm: restrict to pleated}
    Let $M$ be a hyperbolic manifold of finite-volume. Let $c=\sum c_i g_i$ be a reduced rational chain with each $g_i\in\pi_1(M)$ hyperbolic. Then
    for any admissible surface $f:S\to M$ for $c$ relative to the cusps of degree $n(S)$, there is a pleated surface $f':S'\to M$ that is also admissible for $c$ relative to the cusps such that
    $$\frac{-\chi(S')}{2n(S')}\le \frac{-\chi(S)}{2n(S)}.$$
    In particular,
    $$\scl_{M,\Ends(M)}(c)=\inf_S \frac{-\chi(S)}{2n(S)},$$
    where the infimum is taken over relative admissible surfaces $f:S\to M$ that admit pleated representatives.
\end{thm}

The proof mainly relies on the proposition below.
\begin{prop}\label{prop: pleated rep in htpy class}
    Suppose $f:S\to M$ is a map from a compact surface $S$ with boundary to an orientable hyperbolic manifold $M^n$ of finite volume such that 
    \begin{enumerate}
        \item $f$ is incompressible in the sense that no simple loop has null-homotopic image and no essential loop (simple and not boundary parallel) has image homotopic into a cusp of $M$, and
        \item not all boundary components of $S$ are mapped to cusps of $M$.
    \end{enumerate}
    If the genus of $S$ is $g\ge1$, then $f:S\to M$ is homotopic to a pleated surface map.
\end{prop}

We need the following basic fact to prove Proposition~\ref{prop: pleated rep in htpy class}.
\begin{lemma}\label{lemma: centralizer}
    Let $\Gamma$ be a torsion-free discrete subgroup of $\Isom^+(\Hbb^n)$. Suppose there are nontrivial elements $a,a',b\in\Gamma$ such that $b^{-k}a^{-1}$, $a'b^k$, $a^{-1}a'$ lie in a cyclic subgroup for two different values of $k\in\Z$.
    \begin{enumerate}
        \item If $a^{-1}a'$ is a hyperbolic element, then $b$ commutes with $a$ (and similarly with $a'$).
        \item If $a^{-1}a'$ is a parabolic element, then $b$ is also a parabolic element fixing the same point at infinity as $a^{-1}a'$.
    \end{enumerate}
\end{lemma}
\begin{proof}
    Suppose $k\neq k'\in\Z$ are two such values. Then by assumption
    $b^{-k}a^{-1}$ and $a'b^k$ commute, so
    $$b^{-k} (a^{-1}a')b^k=(b^{-k}a^{-1})\cdot (a'b^k)=a'b^k\cdot b^{-k}a^{-1}=a'a^{-1},$$
    and the same holds replacing $k$ by $k'$.
    It follows that
    $$b^{-k} (a^{-1}a')b^k=b^{-k'} (a^{-1}a')b^{k'},$$
    or equivalently, $b^{k-k'}$ commutes with $a^{-1}a'$.

    If $a^{-1}a'$ is hyperbolic, then its centralizer consists of elements in $\Gamma$ that share the same axis, which is infinite cyclic as $\Gamma$ is torsion-free discrete, by considering the (signed) translation along the axis. Hence $b^{k-k'}$ lying in the centralizer implies that the same holds for $b$. We also know that $b^{-k}a^{-1}$ is in the centralizer of $a^{-1}a'$, so it must commute with $b$. It follows that $b$ commutes with $a$.

    If $a^{-1}a'$ is parabolic, then its centralizer consists of parabolic elements in $\Gamma$ that fix the same point at infinity, again since $\Gamma$ is torsion-free discrete. Thus $b$ must be a parabolic element fixing this point at infinity.    
\end{proof}

\begin{proof}[Proof of Proposition~\ref{prop: pleated rep in htpy class}]
    Let $m\ge1$ be the number of boundary components of $S$.
    The argument is inspired by the proof by Calegari \cite[Lemma 3.7]{Cal:sclbook}
    which handles the $m=1$ case when $M$ is closed. 
    So we assume $m\ge2$ below, and a similar proof works for the $m=1$ case.
    
    Realize $S$ topologically as in Figure \ref{fig: pants} with the indicated pants decomposition.
    We claim that after an appropriate modification of the pants decomposition by powers of Dehn twists around the curves $b_i$'s,
    the restriction of $f$ to each pants does not factor through $S^1$, and hence by Thurston's criterion $f$ has a pleated surface representative.
	
 	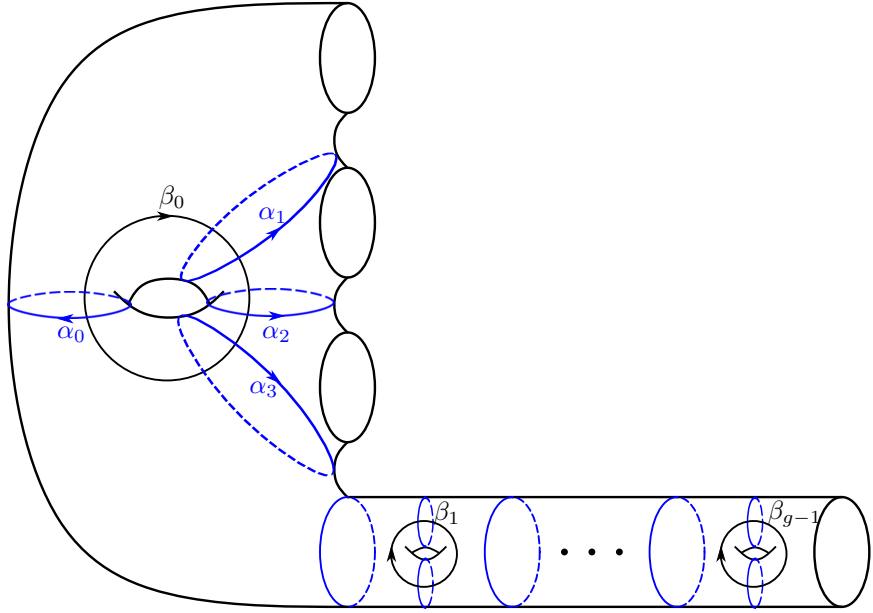
\begin{figure}
 		\labellist
 		\small 
 		\pinlabel $\color{blue}\alpha_0$ at 35 150
 		\pinlabel $\color{blue}\alpha_1$ at 145 215
 		\pinlabel $\color{blue}\alpha_2$ at 147 150
 		\pinlabel $\color{blue}\alpha_3$ at 140 120
 		\pinlabel $\beta_0$ at 90 225
 		\pinlabel $\beta_1$ at 240 52
 		\pinlabel $\beta_{g-1}$ at 430 52
 		\endlabellist
 		\centering
 		\includegraphics[scale=0.7]{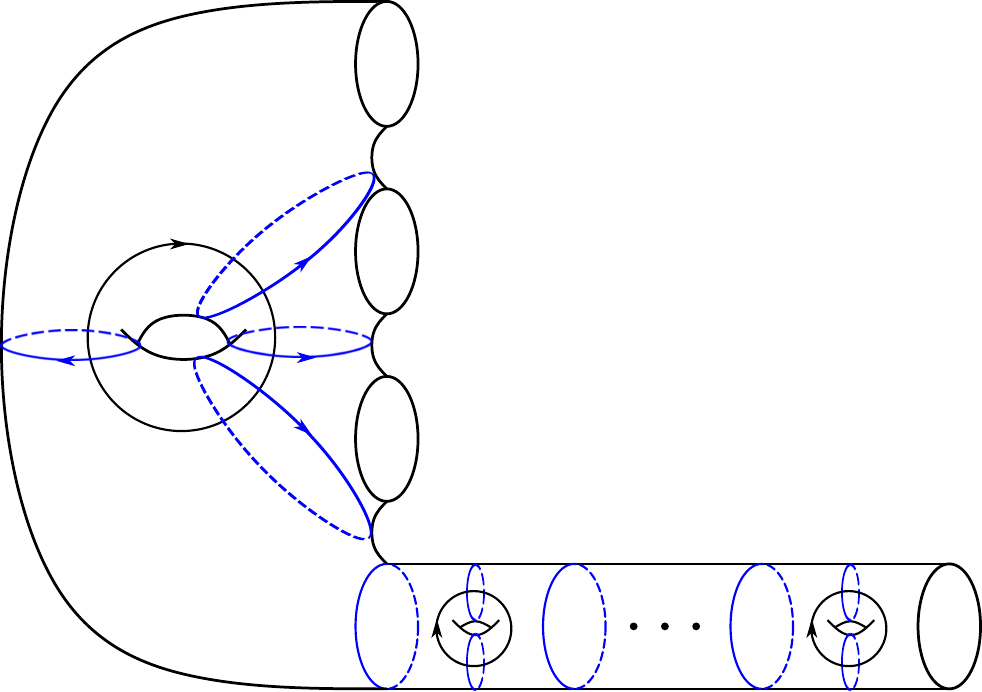}
 		\caption{A pants decomposition of $S$ for the case $m=4$.}\label{fig: pants}
 	\end{figure}
	
    Place the base point $*$ at the intersection of $\alpha_0$ and $\beta_0$, and denote their image under $f_*:\pi_1(S,*)\to \pi_1(M,*)$ as $a_0$ and $b_0$ respectively. Think of each loop $\alpha_i$ as a loop based at $*$ by following an arc on the $\beta_0$ curve. Let $\alpha_m$ be the based loop which is the concatenation $\beta_0\cdot \alpha_0\cdot \overline{\beta}_0$. Let $a_i\in \pi_1(M)$ be the image of each $\alpha_i$ under $f_*$ for all $0\le i\le m$.
    Then under the map $f$, the images of boundary loops of the pants $P_i$ represent $a_{i-1}^{-1},a_i,a_{i-1}^{-1}a_i$, where $0\le i\le m-1$.
    If we perform the $k$-th power of the Dehn twist along $\beta_0$ to modify the pants $P_i$, then the image of the boundary curves would become $b_0^{-k}a_{i-1}^{-1},a_ib_0^k,a_{i-1}^{-1}a_i$, where $k\in\Z$.
    If for two different values of $k$, the restriction of $f$ on the modified pants factors through $S^1$, then Lemma~\ref{lemma: centralizer} implies that either $b_0$ commutes with $a_i$, or $b_0$ is parabolic. The latter cannot occur by assumption, and the former cannot occur since $[\alpha_i,\beta_0]$ represents the boundary of the once-punctured torus neighborhood of $\alpha_i\cup \beta_0$, which is simple and by assumption cannot have trivial image in $\pi_1(M)$. Thus there is at most one value of $k$ that is bad for the pants $P_i$, for each $0\le i\le m-1$. Hence there is some $k\in\Z$ that is good for all of them.

    For a similar reason, for each $1\le i\le g-1$, the loop $\beta_i$ witnesses two pants, and for each of them there is at most one value of $k$ such that the $k$-th power of the Dehn twist along $\beta_i$ modifies it to a pants on which $f$ factors through $S^1$. Hence there must be a choice of $k$ that makes both pants good for this purpose. This part of the proof works exactly as in the proof of \cite[Lemma 3.7]{Cal:sclbook}. This proves the claim and completes the proof.
\end{proof}

The next lemma handles the genus $0$ case of Theorem~\ref{thm: restrict to pleated}, as it requires a different proof.
\begin{lemma}\label{lemma: pleated surface genus 0 case}
    Let $S$ be a compact connected planar surface with negative Euler characteristic, and let $M$ be a (complete) hyperbolic manifold. Suppose $f:S\to M$ is a continuous map such that $f_*(S)$ has non-cyclic image. Then there is a pants decomposition of $S$ such that the restriction of $f$ to each pair of pants has non-cyclic image. Cutting $S$ along loops in the decomposition with parabolic image, we get a surface $S'$ on which the restriction of $f$ is homotopic to a pleated surface map if $M$ has finite volume.
\end{lemma}
\begin{proof}
    Given the existence of such a pants decomposition on $S$, we have an induced pants decomposition on $S'$, where all loops in the decomposition lying in the interior of $S'$ have hyperbolic image. Map each such loop to the geodesic representative. Since each pair of pants has non-cyclic image, applying Thurston's spinning construction to each pants of the decomposition yields the desired pleated surface map.

    So it remains to show the existence of such a pants decomposition on $S$. We proceed by induction on the number $m\ge3$ of boundary components. There is nothing to show for the base $m=3$.
    Suppose the result holds for planar surfaces with $k\ge3$ boundary components and now assume $m=k+1$.
    Label the boundary components of $S$ as $C_0,\ldots, C_k$.
    Place a base point $*$ on $C_0$ and choose proper arcs $\beta_i$ for $1\le i\le k$ with disjoint interior connecting $C_0$ to other boundary components, we obtain based loops $\alpha_i$ freely homotopic to $C_i$ for each $0\le i\le k$.
    Denote the image of $*$ in $M$ also as $*$ and write $\Gamma=\pi_1(M,*)$
    Let $a_i\in \Gamma$ be the element represented by $f(\alpha_i)$. Then we have
    $$a_0a_1\cdots a_k=id.$$
    
    Note the fact that each nontrivial element $g$ in the torsion-free discrete subgroup $\Gamma$ of $\Isom^+(\Hbb^n)$ is contained in a unique maximal cyclic subgroup: If $g$ is hyperbolic, this is the centralizer $Z(g)$; If $g$ is parabolic, this is a maximal $\mathbb{Z}$ subgroup of $Z(g)\cong \mathbb{Z}^{n-1}$.
    
    Let $A$ be the maximal cyclic subgroup containing $a_0$. If $a_i\in A$ for some $i\ge1$, then up to changing the proper arcs $\beta_j$'s as shown in Figure \ref{fig: planar}, we may assume $a_k\in A$; Algebraically this is letting $a'_j=a_j$ for all $j<i$, $a'_j=a_ia_{j+1}a_i^{-1}$ for all $i\le j<k$ and $a'_k=a_i$.

    
    
	
	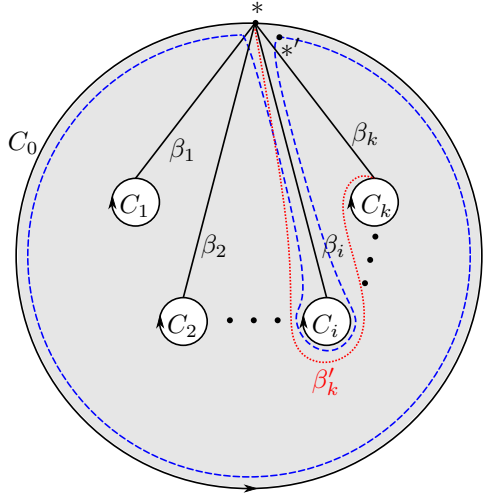
\begin{figure}
		\labellist
		\small 
		\pinlabel $C_0$ at 5 220
		\pinlabel $C_1$ at 75 180
		\pinlabel $C_2$ at 105 105
		\pinlabel $C_i$ at 195 105
		\pinlabel $C_k$ at 227 180
		\pinlabel $\beta_1$ at 105 215
		\pinlabel $\beta_2$ at 125 155
		\pinlabel $\beta_i$ at 200 155
		\pinlabel $\beta_k$ at 220 225
		\pinlabel $\color{red}\beta'_k$ at 195 70
		\pinlabel $*$ at 152 305
		\pinlabel $*'$ at 173 280
		\endlabellist
		\centering
		\includegraphics[scale=0.6]{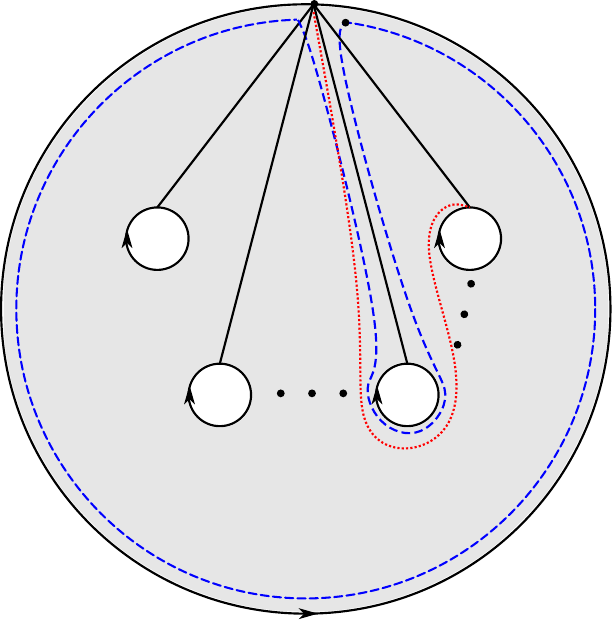}
		\caption{The planar surface $S$ with proper arcs $\beta_i$'s connecting $C_0$ to $C_i$. The dotted red arc $\beta'_k$ indicates the modification that changes $\alpha_{j+1}$ to $\alpha'_{j}$ for all $i\le j<k$ in the special case $j=k-1$. The dashed blue loop cuts out the pair of pants that is the neighborhood of $C_0\cup C_i\cup \beta_i$.}\label{fig: planar}
	\end{figure}
	
    So we assume below either $a_i\notin A$ for all $i\ge1$, or $a_k\in A$.
    We attempt to cut out a pair of pants $P_i$ that is the neighborhood of $C_0\cup C_i\cup \beta_i$, for some $1\le i< k$. For a suitable choice of base point (see Figure \ref{fig: planar}), the fundamental group of the complement $S_i$ has image 
    $$H_i=\langle a_0a_i, a_i^{-1}a_1a_i, \ldots, a_i^{-1}a_{i-1}a_i, a_{i+1},\ldots, a_k \rangle.$$
    
    It suffices to find an index $i$ such that $H_i$ is not cyclic and $a_i\notin A$, in which case $S_i$ has a desired pants decomposition by the induction hypothesis and the union with $P_i$ yields the desired decomposition of $S$.

    Consider the two cases separately.
    If $a_k\in A$, since $f_*(S)$ is not cyclic, there is some $1\le i<k$ such that $a_i\notin A$. We show $H_i$ is non-cyclic. If not, as $H_i$ contains $a_k\in A$, we must have $H_i\subset A$. Then $a_0a_i\in A$, which implies $a_i\in A$ since $a_0\in A$ by definition, contradicting our choice.

    Now suppose $a_i\notin A$ for all $i\ge1$. We show there is at most one $1\le i<k$ such that $H_i$ is cyclic. Indeed, suppose $H_i$ and $H_j$ are both cyclic for some $1\le i<j<k$. Let $A_i$ and $A_j$ be the unique maximal subgroup containing $H_i$ and $H_j$, respectively. Then by definition $a_k\in H_i\le A_i$ and $a_k\in H_j\le A_j$, so we must have $A_i=A_j$. 
    In addition, as $i<j$, we have $a_j\in H_i\le A_i$ and $a_j^{-1}a_ia_j\in A_j$. As $A_i=A_j$, it follows that $a_i,a_j\in A_i=A_j$. Finally, since $a_0a_i\in A_i$ we obtain $a_0\in A_i$, so we must have $A=A_i=A_j$. This contradicts the assumption that $a_i\notin A$.
\end{proof}

Finally we prove Theorem~\ref{thm: restrict to pleated}.
\begin{proof}[Proof of Theorem~\ref{thm: restrict to pleated}]
    The second assertion clearly follows from the first one. To show the first assertion, consider any admissible surface $f:S\to M$ for $c$ relative to the cusps. 

    If $S$ contains any simple loop $\gamma$ whose image under $f$ is null-homotopic, we can cut $S$ along $\gamma$ and glue in two disks to obtain a new admissible surface of the same degree, which does not increase $-\chi^-(S)$. If $S$ contains an essential loop $\gamma$ whose image under $f$ is homotopic to a cusp (i.e. parabolic), we can cut $S$ along $\gamma$, which creates two new boundary components that land in cusps. The new surface is still admissible relative to the cusps, and this does not change the degree or $-\chi^-(S)$.

    After finitely many such steps, we may assume that $f$ is incompressible in the sense of Proposition~\ref{prop: pleated rep in htpy class}. In addition, any component with all boundary components mapped to cusps can be discarded.
    Moreover, since $c$ is reduced, each component of $S$ has negative Euler characteristic.

    Now for each component $\Sigma$ of $S$ that has positive genus, Proposition~\ref{prop: pleated rep in htpy class} implies that $\Sigma$ has a pleated representative.
    For each planar component of $S$, by Lemma~\ref{lemma: pleated surface genus 0 case}, $f$ admits a pleated representative (and no cutting is needed as it is already incompressible).
\end{proof}

\subsection*{Acknowledgments}
Lvzhou thanks OpenAI Codex for suggesting the proof strategy of Lemma~\ref{lemma: pleated surface genus 0 case}, although the proof was manually reorganized and simplified for readability.
LC is partially supported by NSF DMS-2506702.

\bibliographystyle{alpha}
\bibliography{orbifold_gaps}

\begin{thebibliography}{{Tao}16}

\bibitem[Bav91]{bavard}
Christophe Bavard.
\newblock Longueur stable des commutateurs.
\newblock {\em Enseign. Math. (2)}, 37(1-2):109--150, 1991.

\bibitem[BBF16]{BBF}
Mladen Bestvina, Ken Bromberg, and Koji Fujiwara.
\newblock Stable commutator length on mapping class groups.
\newblock {\em Ann. Inst. Fourier (Grenoble)}, 66(3):871--898, 2016.

\bibitem[Bro81]{brooks}
Robert Brooks.
\newblock Some remarks on bounded cohomology.
\newblock In {\em Riemann surfaces and related topics: {P}roceedings of the
  1978 {S}tony {B}rook {C}onference ({S}tate {U}niv. {N}ew {Y}ork, {S}tony
  {B}rook, {N}.{Y}., 1978)}, volume~97 of {\em Ann. of Math. Stud.}, pages
  53--63. Princeton Univ. Press, Princeton, N.J., 1981.

\bibitem[Cal08]{Cal:sclhypmfd}
Danny Calegari.
\newblock Length and stable length.
\newblock {\em Geom. Funct. Anal.}, 18(1):50--76, 2008.

\bibitem[Cal09]{Cal:sclbook}
Danny Calegari.
\newblock {\em scl}, volume~20 of {\em MSJ Memoirs}.
\newblock Mathematical Society of Japan, Tokyo, 2009.

\bibitem[CC25]{CaiClay}
Tommy~Wuxing Cai and Adam Clay.
\newblock Generalized torsion in amalgams, 2025.

\bibitem[CF10]{CF:sclhypgrp}
Danny Calegari and Koji Fujiwara.
\newblock Stable commutator length in word-hyperbolic groups.
\newblock {\em Groups Geom. Dyn.}, 4(1):59--90, 2010.

\bibitem[CFL16]{CFL16}
Matt Clay, Max Forester, and Joel Louwsma.
\newblock Stable commutator length in {B}aumslag-{S}olitar groups and
  quasimorphisms for tree actions.
\newblock {\em Trans. Amer. Math. Soc.}, 368(7):4751--4785, 2016.

\bibitem[CH20a]{CH:sclgapgog_old}
Lvzhou Chen and Nicolaus Heuer.
\newblock Spectral gap of scl in graphs of groups and $3$-manifolds, 2020.
\newblock arXiv preprint 1910.14146v2.

\bibitem[CH20b]{CH:sclgapgog_new}
Lvzhou Chen and Nicolaus Heuer.
\newblock Spectral gap of scl in graphs of groups and $3$-manifolds, 2020.
\newblock arXiv preprint 1910.14146v4.

\bibitem[CH23]{CH:RAAG_chain}
Lvzhou Chen and Nicolaus Heuer.
\newblock Stable commutator length in right-angled {A}rtin and {C}oxeter
  groups.
\newblock {\em J. Lond. Math. Soc. (2)}, 107(1):1--60, 2023.

\bibitem[Che20]{Chen:sclBS}
Lvzhou Chen.
\newblock Scl in graphs of groups.
\newblock {\em Invent. Math.}, 221(2):329--396, 2020.

\bibitem[Heu19]{Heuer}
Nicolaus Heuer.
\newblock Gaps in {SCL} for amalgamated free products and {RAAG}s.
\newblock {\em Geom. Funct. Anal.}, 29(1):198--237, 2019.

\bibitem[IMT19]{gentorsion}
Tetsuya Ito, Kimihiko Motegi, and Masakazu Teragaito.
\newblock Generalized torsion and decomposition of {$3$}--manifolds.
\newblock {\em Proc. Amer. Math. Soc.}, 147(11):4999--5008, 2019.

\bibitem[Nak89]{length_gap}
Toshihiro Nakanishi.
\newblock The lengths of the closed geodesics on a {R}iemann surface with
  self-intersections.
\newblock {\em Tohoku Math. J. (2)}, 41(4):527--541, 1989.

\bibitem[Ser80]{Serre}
Jean-Pierre Serre.
\newblock {\em Trees}.
\newblock Springer-Verlag, Berlin-New York, 1980.
\newblock Translated from the French by John Stillwell.

\bibitem[{Tao}16]{tao}
Jing {Tao}.
\newblock {Effective quasimorphisms on free chains}.
\newblock {\em arXiv e-prints}, page arXiv:1605.03682, May 2016.

\end{thebibliography}

\end{document}